\documentclass[a4paper,twoside,reqno]{amsart}
\usepackage[foot]{amsaddr}

\usepackage{graphicx}
\usepackage{a4wide}
\usepackage{amstext, amsmath, amssymb}
\usepackage{latexsym}
\usepackage{booktabs}
\usepackage{url}
\usepackage{import}
\usepackage{fancyvrb}
\usepackage{stmaryrd}
\usepackage{pdfsync}
\usepackage{algorithm}
\usepackage{tikz}
\usepackage{pgfplots}
\usepackage{pgfplotstable}
\usepackage{subfig}
\usepackage{tabularx}
\usepackage{todonotes}
\usepackage[numbers,square,sort&compress]{natbib}

\newcommand{\RR}{\mathbb{R}}

\newcommand{\bfzero}{\boldsymbol 0}

\newcommand{\mcP}{\mathcal{P}}

\newcommand{\mcK}{\mathcal{K}}
\newcommand{\mcE}{\mathcal{E}}

\newcommand{\mcF}{\mathcal{F}}
\newcommand{\mcA}{\mathcal{A}}

\newcommand{\mcT}{\mathcal{T}}
\newcommand{\mcH}{\mathcal{H}}
\newcommand{\mcO}{\mathcal{O}}
\newcommand{\mcX}{\mathcal{X}}

\newcommand{\tn}{|\mspace{-1mu}|\mspace{-1mu}|}

\newcommand{\jump}[1]{[#1]}
\newcommand{\mean}[1]{\langle {#1} \rangle}
\newcommand{\dgmean}[1]{\left\{\!\!\left\{ {#1} \right\}\!\!\right\}}

\newcommand{\Gammah}{{\Gamma_h}}
\newcommand{\nablas}{\nabla_\Gamma}
\newcommand{\nablash}{\nabla_{\Gamma_h}}

\newcommand{\Ps}{{P}_\Gamma}
\newcommand{\Psh}{{P}_{\Gamma_h}}

\newcommand{\foralls}{\forall\,}
\newcommand{\ds}{\,\mathrm{d} \Gamma}

\newcommand{\dx}{\,\mathrm{d}x}

\newcommand{\dE}{\,\mathrm{d}E}
\newcommand{\dK}{\,\mathrm{d}K}
\newcommand{\dM}{\,\mathrm{d}M}

\DeclareMathOperator{\spann}{span}
\DeclareMathOperator{\dist}{dist}

\DefineVerbatimEnvironment{code}{Verbatim}{frame=single,rulecolor=\color{blue}}

\numberwithin{equation}{section}
\newtheorem{lemma}{Lemma}[section]
\newtheorem{proposition}{Proposition}[section]
\newtheorem{theorem}{Theorem}[section]
\newtheorem{remark}{Remark}[section]
\newtheorem{corollary}{Corollary}[section]

\begin{document}
\title{\bf A Cut Discontinuous Galerkin Method for the Laplace-Beltrami Operator}
\author{Erik Burman}
\address[Erik Burman]{Department of Mathematics, University College London, London, UK--WC1E 6BT, United Kingdom}
\email{e.burman@ucl.ac.uk}

\author{Peter Hansbo}
\address[Peter Hansbo]{Department of Mechanical Engineering, J\"onk\"oping University,
SE-55111 J\"onk\"oping, Sweden.}
\email{Peter.Hansbo@jth.hj.se}

\author{Mats G.\ Larson}
\email{mats.larson@umu.se}
\author{Andr\'e Massing}
\address[Mats G.\ Larson, Andr\'e Massing]{Department of Mathematics and Mathematical Statistics, Ume{\AA} University, SE-90187 Ume{\AA}, Sweden}
\email{andre.massing@umu.se}

\maketitle

%%%%%%%%%%%%%%%%%%%%%%%%%%%%%%%%%%%%%%%%%%%%%%%%%%%%%%%%%%%%%%%%%%%%%%%
\begin{abstract} 
  { We develop a discontinuous cut finite element method (CutFEM) for
    the Laplace-Beltrami operator on a hypersurface embedded in
    $\mathbb{R}^d$. The method is constructed by using a discontinuous
    piecewise linear finite element space defined on a background mesh
    in $\mathbb{R}^d$. The surface is approximated by a continuous
    piecewise linear surface that cuts through the background mesh in an
    arbitrary fashion. Then a discontinuous Galerkin method is
    formulated on the discrete surface and in order to obtain
    coercivity, certain stabilization terms are added on the faces
    between neighboring elements that provide control of the
    discontinuity as well as the jump in the gradient.
    We derive optimal a priori error and condition number estimates
    which are independent of the  positioning of the surface in the background mesh.
    Finally, we present numerical examples confirming
    our theoretical results.}

  \noindent {\tiny KEY WORDS.} 
  surface PDE, Laplace-Beltrami, discontinuous Galerkin, cut finite element method
\end{abstract}

\section{Introduction}

\subsection{Background} 
Recently there has been a rapid development of so-called cut finite
element methods (CutFEM), which provide a technology for
discretization of both the geometry of the computational domain and
the partial differential equations (PDE) which we seek to solve.
The basic idea in CutFEM is to represent the geometry of the domain on
a fixed background mesh which is also used to construct the finite
element space.
The geometric representation of surfaces consists of
elementwise smooth surfaces that are allowed to cut through the
background mesh in an arbitrary fashion. 
The active
background mesh then consists of all elements that are cut by the domain and
the finite element space used in the computation is the restriction to
the active mesh.
The variational formulation of the
PDE
is defined on the cut elements.
Boundary and
interface conditions are enforced weakly. 
In order to obtain a stable
method, independent of the position of the geometry in the background
mesh, and to handle the cut elements in the analysis, certain
stabilization terms are added that provide control of the local
variation of the discrete functions. Further stabilization may be
necessary to control for instance convection or to establish an
inf-sup condition. CutFEM can be used to handle problems in the
bulk domain, on surfaces, and coupled bulk-surface
problems.

\subsection{Earlier Work} 
While the development of standard finite element methods on
triangulated surfaces for the numerical approximation of surface PDEs
was already initiated by the seminal paper by~\citet{Dziuk1988},
CutFEM-type methods for surface PDEs have been introduced only
recently. 
\citet{OlshanskiiReuskenGrande2009} proposed the first discretization of the
Laplace-Beltrami operator based on restricting
continuous piecewise linear finite element functions from the ambient space
to the surface. 
The matrix properties of the resulting system were
investigated in \citet{OlshanskiiReusken2010} showing that diagonal
preconditioning can cures the discrete system from being severely
ill-conditioned.
As an alternative, stabilization techniques based on face stabilization
and full gradient stabilization were introduced and analyzed 
by \citet{BurmanHansboLarson2015} and
\citet{DeckelnickElliottRanner2013}, respectively.
Face stabilization techniques were also used in
\citet{BurmanHansboLarsonEtAl2015}
to develop a cut finite element method for
stationary convection problems on surface,
while \citet{OlshanskiiReuskenXu2014} employed
streamline diffusion techniques.
For evolving surface problems using CutFEM-type techniques,
we refer to \citet{OlshanskiiReuskenXu2014a} and \citet{HansboLarsonZahedi2015b}.
Higher order methods can be found in \citet{GrandeReusken2014}
and adaptive methods including a posteriori error estimates
were presented by \citet{DemlowOlshanskii2012} and \citet{ChernyshenkoOlshanskii2014}.  

For bulk problems we mention the
following contributions: 
Interface problems were considered in \citet{HansboHansbo2002}
with similar techniques being used in \citet{HansboHansboLarson2003} and
\citet{MassingLarsonLoggEtAl2013}
to develop overlapping mesh methods.
Face-based so-called ghost penalties were then 
employed to solve the Poisson boundary problem \citep{BurmanHansbo2012},
the Stokes boundary problem \citep{BurmanHansbo2013,MassingLarsonLoggEtAl2013a}
and  Stokes interface problems \citep{HansboLarsonZahedi2013,WangChen2015}.
Alternative CutFEM schemes 
for the Stokes interface problem can be found
in \citet{GrossReusken2007} where the pressure space were enriched in the vicinity of the interface, 
and in \citet{HeimannEngwerIppischEtAl2013,SollieBokhoveVegt2011} which are
based on unfitted discontinuous Galerkin methods using cell-merging techniques
problems to obtain stable and well-conditioned numerical schemes.
Higher order discontinuous Galerkin
with extended element stabilization for an elliptic problem
were investigated by \citet{JohanssonLarson2013}. 
For coupled surface-bulk problems,
see \citet{BurmanHansboLarsonEtAl2014,GrossOlshanskiiReusken2014}, 
and \citet{MassingLarsonLogg2013} for implementation issues.
 We refer to the overview article \citet{BurmanClausHansboEtAl2014} on
cut finite element methods and the references therein.
Finally, we want to mention~\citet{DziukElliott2013}
for a recent and comprehensive overview on methods for surface PDEs.

\subsection{New Contributions}
In this paper we develop a stable cut discontinuous Galerkin method 
for the Laplace-Beltrami operator on a hypersurface in $\mathbb{R}^d$.
Discontinuous Galerkin methods generally provide excellent conservation properties
and advantageous flexibility in terms of mesh refinement and local space enrichment
and additionally exhibit good stability
properties for convection-diffusion problems and problems involving discontinuous coefficients.
Extending the results from \citet{BurmanHansboLarson2015},
this paper is to the best of our knowledge
the first step in developing CutFEM-type discontinuous
Galerkin method for surface problems.
For triangulated surfaces a discontinuous Galerkin 
method was proposed and analyzed in \citet{DednerMadhavanStinner2013}. 
The analysis of the cut discontinuous Galerkin method on surfaces however 
poses additional challenges and thus requires different technical tools. 

We consider discontinuous piecewise linear elements 
on a background mesh consisting of simplices in $\mathbb{R}^d$. 
The discrete surface is continuous and is on each element given by a plane 
that cuts through the element. Such a surface may, e.g., conveniently be constructed by taking the zero
level set of a continuous piecewise linear approximation of the
distance function. We assume that the discrete surface converges to 
the exact surface in such a way that the error in position and normal 
is $O(h^2)$ and $O(h)$, respectively. On the piecewise linear surface we formulate a
discontinuous Galerkin method.  

In order to prove basic stability results the formulation must be stabilized and we use consistent terms
that provide control of the variation of the finite element functions
over the faces in the background mesh. This is achieved by controlling
the jump in the function as well as the gradient across interior faces
in the active mesh.  Essentially, the stabilization terms enable
control of functions on an element in terms of neighboring elements,
which allows us to handle the presence of small and
deteriorated surface elements. 

Two technical estimates are critical 
to our analysis. First, an inverse estimate of the co-normal fluxes 
at edges in terms of the tangent gradient and the stabilization terms 
and secondly, a discrete Poincar\'e estimate for discontinuous piecewise linear functions that provides control of the $L^2$ norm on the active 
mesh in terms of the tangent gradient and the stabilization terms. 
These results extend the corresponding results in
\citet{BurmanHansboLarson2015} to discontinuous piecewise polynomials
and also provide certain improvements in the details of the proof.
Based on the stability results we derive a priori estimates of the
error in the energy and $L^2$ norm that takes both the approximation
of the solution and the aproximation of the geometry into
account. Furthermore, again using the stabilization, we prove an optimal upper bound for the
condition number. We emphasize that all our results are completely
independent of the relative position of the discrete surface in the background
mesh and only information available from the discrete surface is used
in the definition of the method.

\subsection{Outline}
In Section~\ref{sec:model-problem} we present the
model problem and formulate the numerical method, 
while in Section~\ref{sec:domain-perturbation},
estimates related to the error resulting from the approximation
of the hypersurface are collected.
Section~\ref{sec:approximation-props} is devoted to the construction of an interpolation
operator and the formulation of the necessary interpolation estimates.
Afterwards,
we develop stability estimates for the proposed numerical method 
in Section~\ref{sec:stability-estimates},
including
certain Poincar\'e estimates for finite element spaces on thin domains.
We also prove novel inverse estimates for the co-normal flux accounting for
irregular surface geometry discretizations which are typical in CutFEM methods.
The core of Section~\ref{sec:a-priori-estimates} derives a priori
error estimates based on a Strang Lemma approach, followed by providing
bounds for condition number in Section~\ref{sec:condition-number-estimate}.
Finally, the numerical studies presented Section~\ref{sec:numerical-results}
serves to corroborate our theoretical findings and to illustrate
the importance of the employed stabilization techniques.

\section{Model Problem and Finite Element Method}
\label{sec:model-problem}
\subsection{Preliminaries}
\label{ssec:preliminaries}
In what follows, $\Gamma$ denotes a compact and oriented $C^k$-hypersurface,
$k \geqslant 2$ without boundary which is embedded in ${{\RR}}^{d}$ and equipped with a
normal field $n: \Gamma \to \RR^{d}$ of class $C^{k-1}$.
We let $\rho \in C^k(U_{\delta_0}(\Gamma))$ be the signed distance function
induced by the normal field $n$ and uniquely defined 
on the tubular neighborhood 
$U_{\delta_0}(\Gamma) = \{ x \in \RR^{d} : \dist(x,\Gamma) < \delta_0
\}$ for some $\delta_0 > 0$,
see~\citet{GilbargTrudinger2001}.
For such a tubular neighborhood, $p(x)$ denotes the closest point projection
given by
\begin{align}
  p(x) =  x - \rho(x) n(p(x))
\end{align}
which maps $x \in U_{\delta_0}(\Gamma)$ to the unique point $p(x) \in \Gamma$ such that $| p(x) - x | = \dist(x, \Gamma)$.
The closest point projection allows to extend any function on $\Gamma$ to its
tubular neighborhood $U_{\delta_0}(\Gamma)$ using the pull back
\begin{equation}
\label{eq:extension}
u^e(x) = u \circ p (x)
\end{equation}
In particular, we can smoothly extend the normal field $n_{\Gamma}$ to the tubular neighborhood $U_{\delta_0}(\Gamma)$.
On the other hand, 
for any subset $\widetilde{\Gamma} \subseteq U_{\delta_0}(\Gamma)$ such that 
$p: \widetilde{\Gamma} \to \Gamma $ is bijective, a function $w$ on
$\widetilde{\Gamma}$ can be lifted to $\Gamma$ by the push forward
\begin{align}
  (w^l(x))^e = w^l \circ p = w \quad \text{on } \widetilde{\Gamma}
\end{align}

A function $u: \Gamma \to \RR$ is of class $C^l(\Gamma)$, $l \leqslant k$ if 
there exists an extension $\overline{u} \in C^l(U)$ with $\overline{u}|_{\Gamma} = u$
for some $d$-dimensional neighborhood $U$ of $\Gamma$.
Then the tangent gradient $\nabla_\Gamma$ on $\Gamma$ is defined
by
\begin{equation}
\nablas u = \Ps \nabla \overline{u}
\label{eq:tangent-gradient}
\end{equation}
with $\nabla$ the ${{\RR}}^{d}$ gradient and $\Ps = \Ps(x)$ the
orthogonal projection of $\RR^{d}$ onto the tangent plane of $\Gamma$ at $x \in \Gamma$
given by
\begin{equation}
  \Ps = I - n_{\Gamma} \otimes n_{\Gamma}
\end{equation}
where $I$ is the identity matrix. It can easily been shown that the definition~\eqref{eq:tangent-gradient} is independent of the extension $\overline{u}$.
We let $\| w \|^2_\Gamma = (w,w)_\Gamma $ denote the $L^2(\Gamma)$ norm on $\Gamma$ and
introduce the Sobolev $H^m(\Gamma)$ space as the subset of
$L^2$ functions for which the norm
\begin{equation}
\| w \|^2_{m,\Gamma} = \sum_{k=0}^m \| D^{P,k}_\Gamma w\|_\Gamma^2,
\quad m = 0,1,2
\end{equation}
is defined. Here, the $L^2$ norm for a matrix is based on the pointwise
Frobenius norm, $D^{P,0}_\Gamma w = w$ and the derivatives $D^{P,1}_\Gamma = \Ps\nabla w,
D^{P,2}_\Gamma w = \Ps(\nabla \otimes \nabla w)\Ps$ are taken in a weak sense.
Finally, for any function space $V$ defined on $\Gamma$, we denote
the space consisting of extended functions by $V^e$ and
correspondingly, we use the notation $V^l$ to refer to the lift of a
function space~$V$ defined on $\widetilde{\Gamma}$.

\subsection{The Continuous Problem}
We consider the following problem: find $u: \Gamma \rightarrow {{\RR}}$
such that
\begin{align}
  \label{eq:LB}
-\Delta_\Gamma u = f \quad \text{on $\Gamma$}
\end{align}
where $\Delta_\Gamma$ is the Laplace-Beltrami operator on $\Gamma$
defined by
\begin{equation}
\Delta_\Gamma = \nabla_\Gamma \cdot \nabla_\Gamma
\end{equation}
and $f\in L^2(\Gamma)$ satisfies $\int_\Gamma f = 0$. The
corresponding weak statement takes the form: find
$u \in H^1(\Gamma)/\RR$
such that
\begin{equation}
a(u,v) = l(v) \quad \forall v \in H^1(\Gamma)/\RR
\end{equation}
where
\begin{equation}
a(u,v) = (\nabla_\Gamma u, \nabla_\Gamma v)_\Gamma, \quad l(v) = (f,v)_\Gamma
\end{equation}
and $(v,w)_\Gamma = \int_\Gamma v w$ is the $L^2$ inner product.
It follows from the Lax-Milgram lemma that this problem has a unique
solution. For smooth surfaces we also have the elliptic regularity
estimate
\begin{equation}
\| u \|_{2,\Gamma} \lesssim \|f\|_\Gamma
\label{eq:ellreg}
\end{equation}

Here and throughout the paper we employ the notation $\lesssim$ to denote less
or equal up to a positive constant that is always independent of the
mesh size. The binary relations $\gtrsim$ and $\sim$ are defined analogously. 

\subsection{The Cut Finite Element Space}
\label{ssec:domain-and-fem-spaces}

Let $\widetilde{\mcT}_{h}$ be a quasi uniform mesh, with mesh parameter
$0<h\leq h_0$, into shape regular simplices of a open and bounded domain $\Omega$ in
$\RR^{d}$ containing $U_{\delta_0}(\Gamma$). 
On $\widetilde{\mcT}_h$, let $\rho_h$ be an continuous, piecewise
linear approximation of the distance function $\rho$
and define the discrete surface $\Gamma_h$ as the zero level set of
$\rho_h$; that is
\begin{equation}
\Gamma_h = \{ x \in \Omega : \rho_h(x) = 0 \}
\end{equation}

We note that $\Gamma_h$ is a polygon with flat faces and we let
$n_h$ be the piecewise constant exterior unit normal to $\Gamma_h$.
We assume that: 
\begin{itemize}
\item $\Gamma_h \subset U_{\delta_0}(\Gamma)$ and that the closest
point mapping $p:\Gamma_h \rightarrow \Gamma$ is a bijection for $0< h
\leq h_0$.
\item The following estimates hold
\begin{equation} \| \rho \|_{L^\infty(\Gamma_h)} \lesssim h^2, \qquad
\| n^e - n_h \|_{L^\infty(\Gamma_h)} \lesssim h
\label{eq:geometric-assumptions-II}
\end{equation} 
\end{itemize}
These properties are, for instance, satisfied if
$\rho_h$ is the Lagrange interpolant of $\rho$.
For the mesh $\widetilde{\mcT}_{h}$, we define active background
mesh
\begin{align} 
  \mcT_h &= \{ T \in \widetilde{\mcT}_{h} : T \cap \Gamma_h \neq \emptyset \}
  \label{eq:narrow-band-mesh}
  \\
  \intertext{and its set of interior faces by}
  \mcF_h &= \{ F =  T^+ \cap T^-: T^+, T^- \in \mcT_h \}
\end{align}
The face normals $n^+_F$ and $n^-_F$ are then given by the unit
normal vectors which are perpendicular on $F$ and are pointing exterior
to $T^+$ and $T^-$, respectively.
The corresponding collection of geometric entities for the
surface approximation $\Gamma_h$ are denoted by
\begin{align}
  \mcK_h&=\{K = \Gamma_h \cap T : T \in \mcT_h \}
  \\
  \mcE_h&=\{E = K^+ \cap K^-:  K^+, K^- \in \mcK_h \}
\end{align}
To each interior edge $E$ we associate the normals $n^{\pm}_E$ 
given by the unique unit vector which is
coplanar to the surface element $K^{\pm}$, perpendicular
to $E$ and points outwards with respect to $K^{\pm}$. 
Note that while the two normals $n_F^{\pm}$ only differ by a sign,
the normals $n_E^{\pm}$ do lie in genuinely different planes.
The various set of geometric entities are illustrated in
Figure~\ref{fig:domain-set-up}.
\begin{figure}[htb]
  \begin{center}
    \includegraphics[width=0.45\textwidth]{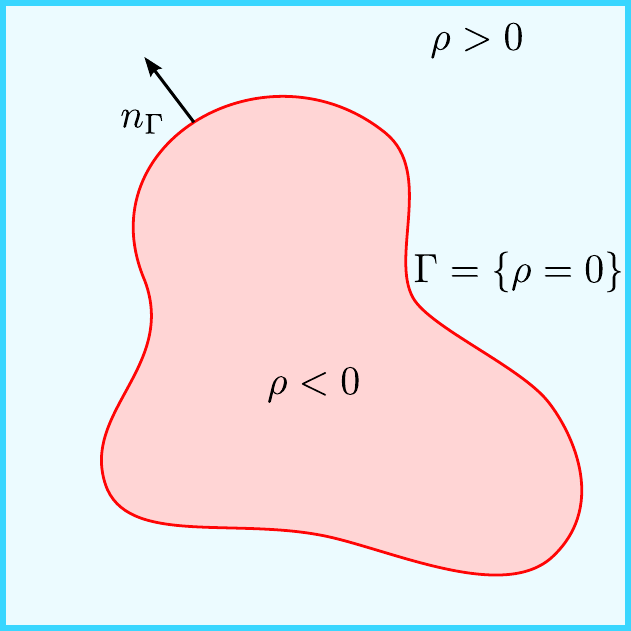}
    \hspace{0.03\textwidth}
    \includegraphics[width=0.45\textwidth]{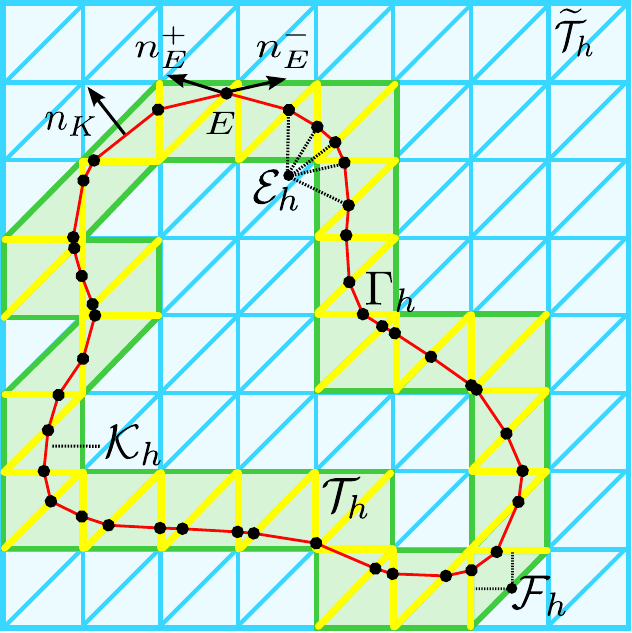}
  \end{center}
  \caption{Domain set-up}
  \label{fig:domain-set-up}
\end{figure}
We observe that the active background mesh $\mcT_h$ gives raise to a discrete or approximate
$h$-tubular neighborhood of $\Gamma_h$, which we denote by
\begin{align}
 N_h = \cup_{T \in \mcT_h} T
\end{align}
Note that for all elements
$T \in \mcT_h$ there is a neighbor $T'\in \mcT_h$ such that $T$ and
$T'$ share a face. 

Finally, let
\begin{equation}
V_h = \bigoplus_{T \in \mcT_h} P_1(T)
\label{eq:Vh-def}
\end{equation}
be the space of discontinuous piecewise linear polynomials defined
on $\mcT_h$ and define the subspace
%of $v
\begin{align}
  V_{h,0} = \{ v \in V_h : \lambda_{\Gamma_h}(v) = 0 \}
\label{eq:Vh0-def}
\end{align}
consisting of those $v \in V_h$ with zero average $\lambda_{\Gamma_h}(v) =
\int_{\Gamma_h} v$. Note that $V_{h,0}$
can be considered as the discrete version of $H^1(\Gamma)/\RR$.

\subsection{Discontinuous Galerkin Method}
\label{ssec:dg-method}
Before we formulate the cut discontinuous Galerkin method for the
Laplace-Beltrami problem~\eqref{eq:LB},
we introduce the notation of average and jump first. 
As the co-normal vectors $n_E^{\pm}$ are generally not collinear,
the average flux for a vector-valued, piecewise continuous function
$\sigma$ on $\mcK_h$ is defined by
\begin{align}
\mean{\sigma}|_E &= \sigma_E^{+} \cdot n^{+}_E - \sigma_E^{-} \cdot n^{-}_E
  \label{eq:mean-def}
\end{align}
while for any, possibly vector-valued, piecewise continuous function $w$ on
$\mcK_h$, the jump across an interior edge $E \in
\mcE_h$ is defined by
\begin{align}
\jump{w}|_E &= w_E^{+} - w_E^{-} 
\end{align}
with
$w(x)^\pm = \lim_{t\rightarrow 0^+} w(x - t n_E^{\pm})$.
Similarly, the jump across an interior face $F\in \mcF_h$ is
given by
\begin{align} 
\jump{w}|_F &= w_F^{+} - w_F^{-}.
\end{align}

Next, we define the bilinear form $A_h(\cdot,\cdot)$ by
\begin{equation}
A_h(v,w) = a_h(v,w) +  j_{h}(v,w) \quad \forall v,w \in V_h
\end{equation}
with
\begin{align}
  a_h(v,w) &= \sum_{K \in \mcK_h} (\nablash v, \nablash w)_{K}
  \nonumber
  \\ 
  &\qquad 
  - \sum_{E \in \mcE_h} (\mean{\nabla v}, \jump{w})_E
  - \sum_{E \in \mcE_h} (\jump{v}, \mean{\nabla w})_E
  \nonumber
  \\
  &\qquad + \sum_{E \in \mcE_h} \beta_E h^{-1}(\jump{v},\jump{w})_{E}
  \label{eq:ah-def}
  \\
  j_{h}(v,w) &= \sum_{F \in \mcF_h} 
  \beta_F h^{-2}(\jump{v},\jump{w})_{F}
  + \gamma (n_F\cdot \jump{\nabla v}, n_F \cdot \jump{\nabla w})_{F}
  \label{eq:jh-def}
\end{align}
where $\beta_E,\beta_F,\gamma$ are positive parameters.
The right hand side is given by linear form
\begin{equation}
l_h(v) = ( f^e, v )_{\Gamma_h}
\end{equation}
Then the cut discontinuous Galerkin finite element method for the
Laplace-Beltrami problem~\eqref{eq:LB}
takes the form: find  $u_h \in V_{h,0}$ such that
\begin{equation}
  A_h(u_h,v) = l_h(v) \quad \forall v \in V_{h,0}
\label{eq:weak-cutfem-formulation}
\end{equation}

We conclude this section by introducing suitable norms for the forthcoming stability
and error analysis. Here and throughout this work, any norm $\| \cdot
\|_{\mcP_h}$ involving a collection
of geometric entities $\mcP_h$ is defined in the usual way, i.e., by summing over all
entities: $\| \cdot \|_{\mcP_h}^2 = \sum_{P \in \mcP_h} \| \cdot
\|_{P}^2$.
For the discrete energy norm associated to the weak
formulation~\eqref{eq:weak-cutfem-formulation} we chose
\begin{align}
  \label{eq:discrete-energy-norm}
  \tn v \tn_h^2 &= \| \nablash v \|^2_{\mcK_h}
  + \|h^{-1/2} [v] \|^2_{\mcE_h}
  +  \|h^{-1} [v] \|^2_{\mcF_h}
  +\| n_F \cdot [\nabla v] \|^2_{\mcF_h}
\end{align}
and for the function space $(H^2(\Gamma))^e + V_h$, we define
\begin{align}
  \tn v \tn_{\ast,h}^2 =  \tn v \tn_h^2 + \|h^{1/2} \mean{\nablash v} \|_{\mcE_h}^2
\end{align}
For convenience, the part of the energy norm associated
with the face terms is summarized by the notation
\begin{equation}
  \tn v \tn^2_{\mcF_h} 
  = \|h^{-1} [v] \|^2_{\mcF_h}
  +\|n_F \cdot [\nabla v] \|^2_{\mcF_h}
  \label{eq:face-norm}
\end{equation}
Although it is not obvious from the definition, we will later prove in that the discrete energy 
norm~\eqref{eq:discrete-energy-norm} indeed defines a norm on
$V_{h,0}$, see Corollary~\ref{eq:discrete-poincare-Nh}.
\begin{remark}
  We point that thanks
  to inverse estimate~\eqref{eq:inverse-estimate-cut-on-E},
  the edge-based penalty term
  appearing in~\eqref{eq:ah-def} can be controlled by its face-based
  counterpart in the stabilization form~\eqref{eq:jh-def}.
  More precisely it holds 
  \begin{align}
    \| h^{-1/2} \jump{u} \|_{\mcE_h}
    \lesssim 
    \| h^{-1} \jump{u} \|_{\mcF_h}
  \end{align}
  and consequently, one might simplify formulation~\eqref{eq:ah-def}
  by setting $\beta_E = 0$ at the expanse of possibly requiring a larger value for $\beta_F$.
  \label{rem:simplified-ah}
\end{remark}

\subsection{Lifted Discontinuous Galerkin Form}
\label{ssec:lifted-dg-form}
The mismatch of the smooth surface $\Gamma$ and its discrete
counterpart $\Gamma_h$ gives raise to a geometric variational crime
in the proposed discretization scheme which must be
accounted for in the forthcoming a priori analysis.
An important instrument in assessing the error introduced by this
variational crime will be the following lifted version
of the bilinear form~\eqref{eq:ah-def}:
\begin{align}
a_h^l(v,w) &= \sum_{K^l \in \mcK^l_h} (\nablas v, \nablas w)_{K^l}
\\ \nonumber
&\qquad
- \sum_{E^l \in \mcE^l_h} (\mean{\nabla v}, \jump{w})_{E^l}
- \sum_{E^l \in \mcE^l_h} (\jump{v}, \mean{\nabla w} )_{E^l}
\\ 
&\qquad + \sum_{E^l \in \mcE^l_h} \beta_E h^{-1}(\jump{v},\jump{w})_{E^l}
\label{eq:lifted-dg-form}
\end{align}
where, referring to the notation in
Section~\ref{ssec:dg-method}, the average $\langle
\cdot \rangle$ is taken with respect to $n_{E^l}^{\pm}(x) \in
T_xK^{l,\pm}$; that is the unique co-normal vector which points outwards $K^{l,\pm}$
and is orthogonal to both the surface $n_{\Gamma}(x)$ and the tangential
space $T_xE^{l,\pm}$. 
With the definition of the lifted discontinuous Galerkin form,
it is clear that the
solution~$u$ of problem~\eqref{eq:LB} satisfies 
the following weak problem:
\begin{equation}
  a_h^l(u,v) = l(v) \quad \forall v \in V_h^l
  \label{eq:weakexact}
\end{equation}
For $v \in H^2(\Gamma) + V_h^l$,
the natural energy norm associated with this weak form is given by
\begin{align}
  \label{eq:discrete-energy-norm-lifted}
  \tn v \tn_{\ast}^2 &= \| \nablas v \|^2_{\mcK_h^l}
  + \|h^{-1/2} [v] \|^2_{\mcE_h^l}
  + \|h^{1/2} \langle \nablas  v \rangle \|^2_{\mcE_h^l}
\end{align}

\section{Domain Perturbation Related Estimates}
\label{sec:domain-perturbation}
The aim of this section is to collect the appropriate geometric
estimates which are necessary to quantify the error introduced by
the geometric variational crime, i.e.,
the piecewise linear approximation of the smooth surface $\Gamma$.
While developing face and edge-related estimates simultaneously, 
we essentially follow the presentations in
\citet{Dziuk1988,OlshanskiiReuskenGrande2009,DziukElliott2013}
and include without detailed proofs those results which are well known.
Here and throughout the remaining work, 
we write  $(\cdot, \cdot)_{\RR^d}$ and $\|\cdot \|_{\RR^d}$
to denote the standard scalar product and its associate norm in $\RR^d$. 

\subsection{Gradient of Lifted and Extended Functions}
Derivatives of extend and lifted functions can easily be computed by the chain rule, once
the derivative of the closest point projection is known.
Using the fact that $\nabla \rho = n_{\Gamma}$,
the derivative of the closest point projection computes to
\begin{equation}
Dp = I - n_{\Gamma} \otimes n_{\Gamma} - \rho \nabla \otimes \nabla \rho
\label{eq:derivative-closest-point-projection}
\end{equation}
where $\mcH = \nabla \otimes \nabla \rho = \nabla n_{\Gamma}$ denotes the Hessian of the signed distance function.
The symmetry of the Hessian and the projection $\Ps$ together with
simple fact that $0 = \nabla \| n_{\Gamma} \|_{\RR^d}^2 = 2 \nabla n_{\Gamma}^{T} n_{\Gamma}$ implies
that $\mcH n_{\Gamma} = 0$ and $(n_{\Gamma} \otimes n_{\Gamma}) \mcH = 0$ and therefore that $\mcH
\Ps = \mcH = \Ps \mcH$. 
This allows to
rewrite~\ref{eq:derivative-closest-point-projection} as
\begin{align}
Dp = \Ps (I - \rho \mcH) = \Ps - \rho \mcH
\label{eq:derivative-closest-point-projection-II}
\end{align}
and thus
\begin{align}
  Dv^e = D(v \circ p) = Dv Dp = Dv P_{\Gamma}(I - \rho \mcH)
\end{align}
Exploiting the self-adjointness of $\Ps$, $\Psh$, and $\mcH$,
we have for any vector $a \in \RR^{d}$
\begin{align}
  (
  \nablash v^e, a 
  )_{\RR^d}
  \nonumber
&=
  (
  \nabla v^e, \Psh a 
  )_{\RR^d}
  \nonumber
\\
&= Dv^e \Psh a 
= Dv \Ps (I - \rho \mcH) \Psh a
= (\nabla v,  \Ps(I - \rho \mcH) \Psh a
\\
&= ( \Psh (I - \rho \mcH) \Ps \nabla v, a )_{\RR^d}
\end{align}
that is
\begin{align}
  \nablash v^e = B^{T} \nablas v
\end{align}
where we introduced the invertible linear mapping
\begin{align}
  B = P_{\Gamma}(I - \rho H) P_{\Gamma_h}: T_x(K) \to T_{p(x)}(\Gamma)
  \label{eq:B-def}
\end{align}
which maps the tangential space of $K$ at $x$ to the tangential space of $\Gamma$ at
$p(x)$. Setting $v = w^l$ and using the identity $(w^l)^e = w$, we immediately get that
\begin{align}
  \nablas w^l = B^{-T} \nablash w
\end{align}
for any elementwise differentiable function $w$ on $\Gamma_h$ lifted to $\Gamma$.
From its definition~\eqref{eq:B-def}, it is easy to derive norm bounds for
operators involving $B$ once $\mcH$ can be controlled.
We recall from \cite[Lemma 14.7]{GilbargTrudinger2001}
that for $x\in U_{\delta_0}(\Gamma)$, the Hessian $\mcH$
admits a representation
\begin{equation}\label{Hform}
  \mcH(x) = \sum_{i=1}^d \frac{\kappa_i^e}{1 + \rho(x)\kappa_i^e}a_i^e \otimes a_i^e
\end{equation}
where $\kappa_i$ are the principal curvatures with corresponding
principal curvature vectors $a_i$.
Thus
\begin{equation}
  \|\mcH\|_{L^\infty(U_{\delta_0}(\Gamma))} \lesssim 1
  \label{eq:Hesse-bound}
\end{equation}
for $\delta_0 > 0$ small enough and as an almost
immediate consequence, we have the
following estimates summarized in
\begin{lemma} It holds
  \label{lem:BBTbound}
  \begin{equation}
    \| B \|_{L^\infty(\Gamma_h)} \lesssim 1,
    \quad \| B^{-1} \|_{L^\infty(\Gamma)} \lesssim 1,
    \quad
    \| P_\Gamma - B B^T \|_{L^\infty(\Gamma)} \lesssim h^2
    \label{eq:BBTbound}
  \end{equation}
\end{lemma}
\begin{proof}
  The proof is standard, see~\citet{DziukElliott2013}, and only
  sketched here for completeness.
  The first two bounds follow directly from~\eqref{eq:B-def} and
  ~\eqref{eq:Hesse-bound} together with the assumption
  $\|\rho\|_{L^{\infty}(\Gamma_h)} \lesssim h^2$.
  Similarly, it follows that
  $
  \Ps - B B^T = \Ps - \Ps \Psh \Ps + O(h^2).
  $
  An easy calculation now shows that
  $\Ps - \Ps \Psh \Ps = \Ps (\Ps - \Psh)^2 \Ps$ 
  from which the desired bound follows by observing that
  $
  \Ps - \Psh = 
  (n - n_h)  \otimes n
  +
  n_h \otimes (n - n_h)
  $ and thus 
  $ 
  \| (\Ps - \Psh)^2 \|_{L^{\infty}(\Gamma_h)} \lesssim
  \| n - n_h \|_{L^{\infty}(\Gamma_h)}^2 \lesssim h^2
  $
\end{proof}

\subsection{Change of the Integration Domain}
Next, we derive estimates for the change of the Riemannian measure
when surface and edge integrals are lifted from the discrete surface
to the continuous surface and vice versa.
For this purpose, we define the quotient of measures
$|B|_{d-1}$ and $|B|_{d-2}$ by 
\begin{align}
  \dK^l = |B|_{d-1} \dK, \quad
  \dE^l = |B|_{d-2} \dE 
\end{align}
where $\dK$ and $\dE$ denotes the measure on 
$K \in \mcK_h$ and $E \in \mcE_h$, respectively,
while the corresponding measures on the lifted
entities $K^l \in \mcK_h^l$ and $E^l \in \mcE_h^l$
are denoted by $\dK^l$ and $\dE^l$, respectively.
Picking an element $K$ and one of its boundary edges $E$
together with its outer co-normal $n_E$,
we can assume (after some rigid motion)
that $K \subset \RR^{d-1} \times \{0\}$ and
$E \subset \RR^{d-2} \times \{0,0\}$ and
that $n_{K} = e_{d}$ and $n_E = e_{d-1}$.
In this coordinate system, we have $\dK = \dx_1 \dots \dx_{d-1}$ and 
$\dE = \dx_1 \dots \dx_{d-2}$.

Recall that for any smooth parametrized $k$-dimensional sub-manifold
$M \subset \RR^{d}$
together with a smooth parametrization $p : U \subset \RR^k \to
\RR^{d}$  satisfying $p(U) = M$,
the integral $\int_{M} f \dM$
of a function $f : M \to \RR$
is given by
\begin{align}
  \int_M f \dM = \int_U f(p(x)) \sqrt{\det(g_{ij}(x))}\dx,
  \label{eq:riemannian-integral}
\end{align}
where $g_{ij} = ( \partial_i p, \partial_j p )_{\RR^d}$ $i,j = 1,
\dots,k$ is the first fundamental form of $M$ in local
coordinates. Considering the cases $M = K^l$, $U = K$ and $M = E^l$ and $U = E$
with the closest point projection $p$ as parametrization, we see that
$|B|_{d-1}$ and $|B|_{d-2}$ are precisely given by the $d-1$ and $d-2$ volumes
\begin{align}
  | B |_d = \sqrt{\det((g_{ij}))_{i,j=1}^{d-1}}
  \label{eq:element-measure-distortion}
  \\
  |B|_{d-2} = \sqrt{\det((g_{ij}))_{i,j=1}^{d-2}}
  \label{eq:edge-measure-distortion}
\end{align}
Instead of calculating these volumes exactly, we derive a simple asymptotic
representation in $h$.
Combining the
representation~\eqref{eq:derivative-closest-point-projection-II} of
$Dp$ with the fact that the choice of our coordinate system implies 
\begin{align}
  n_{\Gamma}^i = ( n_{\Gamma}, e_i )_{\RR^d} \lesssim h, 
  \quad i = 1,\dots, d-1, \text{ and thus} \quad
  n_{\Gamma}^d \sim 1 + O(h^2)
\end{align}
allows us to derive an asymptotic expression for the coefficients
$g_{ij}$ for $i,j= 1,\dots,d-1$:
\begin{align}
  g_{ij}  &= 
  (
 (P_{\Gamma} - \rho \mcH)e_i, 
 (P_{\Gamma} - \rho \mcH)e_i
 )_{\RR^d}
 \\
 &=
 ( P_{\Gamma} e_i, P_{\Gamma} e_i )_{\RR^d}+ O(h^2)
\\
&=
 \delta_{ij}  - n_{\Gamma}^i n_{\Gamma}^j + O(h^2)
\\
 &\sim
\delta_{ij} + O(h^2).
\end{align}
Consequently, $\sqrt{\det(g_{ij})} \sim \sqrt{1 + O(h^2)} \sim 1 + O(h^2)$ and thus
\begin{align}
  | B |_d \sim 1 + O(h^2)
  \label{eq:element-measure-distortion-estimate}
  \\
  | B |_{d-1} \sim 1 + O(h^2)
  \label{eq:edge-measure-distortion-estimate}
\end{align}
which leads directly to the bounds summarized in the next Lemma.
\begin{lemma} The following estimates hold
  \label{lem:detBbounds}
  \begin{alignat}{5}
  \| 1 - |B|_{d-1} \|_{L^\infty(\mcK_h)} 
  &\lesssim h^2, 
  & &\qquad
  \||B|_d\|_{L^\infty(\mcK_h)} 
  &\lesssim 1, 
  & &\qquad
  \||B|_d^{-1}\|_{L^\infty(\mcK_h)} 
  &\lesssim 1
    \label{eq:detBbound}
    \\
  \| 1 - |B|_{d-2} \|_{L^\infty(\mcE_h)} 
  &\lesssim h^2, 
  & &\qquad
  \||B|_{d-2}\|_{L^\infty(\mcE_h)} 
  &\lesssim 1, 
  & &\qquad
  \||B|_d^{-1}\|_{L^\infty(\mcE_h)} 
  &\lesssim 1
  \label{eq:detBEbound}
\end{alignat}
\end{lemma}
We point out that
estimates similar to~\eqref{eq:detBEbound} 
were derived in the finite volume context in
\citet{GiesselmannMueller2014} and
used in \citet{DednerMadhavanStinner2013}.

Next, we note that by combining
the estimates for the operator norms~\eqref{eq:BBTbound} and
for the metric distortion
factors~\eqref{eq:detBbound}--\eqref{eq:detBEbound},
the following norm equivalences can be obtained:
\begin{lemma}
  Let $v \in L^2(\mcP_h)$ and $w \in L^2(\mcP_h^l)$
  for $\mcP_h \in \{ \mcK_h, \mcE_h \}$.
  Then it holds
  \label{lem:norm-equivalences}
  \begin{equation}
    \| v^l \|_{\mcP_h^l}
    \sim \| v \|_{\mcP_h}, 
    \qquad \| w \|_{\mcP_h^l} \sim
    \| w^e \|_{\mcP_h}
  \label{eq:norm-equivalences-l2}
  \end{equation}
  If in addition $v \in H^1(\mcK_h)$ and $w \in H^1(\mcK_h^l)$
  the following equivalences are satisfied
  \begin{equation}
    \| \nabla_\Gamma v^l \|_{\mcK_h^l} \sim \| \nablash v \|_{\mcK_h},
    \qquad 
    \| \nablas w \|_{\mcK_h^l} \sim
    \| \nablash w^e \|_{\mcK_h}
  \label{eq:norm-equivalences-h1}
  \end{equation}
\end{lemma}

Before we proceed, let us note that the measure quotients
~\eqref{eq:element-measure-distortion}--\eqref{eq:edge-measure-distortion-estimate}
can be equivalently expressed using the exterior product of the
tangential vectors $\partial_i p$, $i = 1,\dots,d-1$ and the surface
normal $n_{\Gamma}$.  More precisely,
\begin{align}
  | B |_{d-1} &= \| \partial_1 p \wedge \ldots \wedge \partial_{d-1} p \|
  \label{eq:element-measure-distortion-II}
  \\
  | B |_{d-2} &= \| \partial_1 p \wedge \ldots \partial_{d-2} p \wedge
  n_{\Gamma} \|
  \label{eq:edge-measure-distortion-II}
\end{align}
where for given vectors $v_1, \dots, v_{d-1} \in \RR^{d}$
the outer product $ v_1 \wedge \ldots \wedge v_{d-1}$
is the unique vector satisfying
\begin{align}
  \det(v_1,\ldots,v_{d-1},w) = 
  ( v_1 \wedge \ldots \wedge
  v_{d-1}, w 
  )_{\RR^d}
  \quad \forall w \in \RR^{d}.
  \label{eq:def-exterior-product}
\end{align}
With this in mind we can now easily
estimate the deviation of the lifted co-normal vector $n_E^l$
from the co-normal vector $n_{E^l}$ associated with the lifted edge $E^l$.
\begin{lemma}
    \label{lem:cornormal-lifted-bound}
It holds
  \begin{gather}
    \| B^l n_{E}^l - |B|_{d-2} n_{E^l} \|_{L^{\infty}(E^l)} \lesssim h^2,
    \label{eq:cornormal-lifted-bound}
  \end{gather}
\end{lemma}
\begin{proof}
  Expanding the scaled co-normal and lifted discrete co-normal
  vectors in the standard orthonormal basis $e_1, \ldots, e_{d}$,
  it suffices to show that the resulting expansion coefficients satisfy the
  estimates
  \begin{align}
    (B^l n_E^l - |B|_{d-2}^l n_{E^l}, e_i )_{\RR^d}\lesssim h^2 
    \quad i = 1,\ldots, d.
  \end{align}
  To compute $( B^l n_E^l, e_i )_{\RR^d}$, we simply calculate
  \begin{align}
    ( B^l n_E^l, e_i )_{\RR^d}
    &= 
    ( \Ps (I - \rho \mcH) e_{d-1}, e_i )_{\RR^d}
    =
    ( \Ps e_{d-1}, e_i )_{\RR^d} + O(h^2)
    \\
    &= \delta_{d-1,i}  - n_{\Gamma}^{d-1} n_{\Gamma}^i + O(h^2)
    \\
    &=
    \begin{cases}
      O(h^2),
      \quad i = 1,\ldots d-2
      \\
      1 + O(h^2),
      \quad i = d-1
      \\
      -n_{\Gamma}^{d-1} n_{\Gamma}^{d} + O(h^2),
      \quad i = d 
    \end{cases}
    \label{eq:discrete-conormal-asymptotics}
  \end{align}
  Now turning to
  $ -(|B|_{d-2}^l n_{E^l}, e_i )_{\RR^d}$, 
  we observe that due to~\eqref{eq:edge-measure-distortion-II},
  the fact that $-n_E^l$ is pointing inwards and 
  $n_E \perp \partial_1 p,\ldots, \partial_{d-2} p, n_{\Gamma}$, we have
  \begin{align}
    -|B|_{d-2} n_{E^l} = 
    \partial_1 p 
    \wedge
    \dots
    \wedge
    \partial_{d-2} p 
    \wedge n_{\Gamma},
  \end{align}
  and therefore by definition of the exterior product~\eqref{eq:def-exterior-product}
  \begin{align}
    - ( |B|_{d-2} n_{E^l}, e_i )_{\RR^d}
    &= 
    \det(D p e_1, \ldots, D p e_{d-2}, n_{\Gamma}, e_i)
    \\
    &=
    \det(\Ps e_1, \ldots, \Ps e_{d-2}, n_{\Gamma}, e_i) + O(h^2)
    \\
    &=
    \det(e_1, \ldots, e_{d-2}, n_{\Gamma}, e_i) + O(h^2)
    \\
    &=
    \label{eq:contnuous-conormal-asymptotics}
    \begin{cases}
      0 + O(h^2), \quad i = 1,\ldots, d-2, 
      \\
      - n_{\Gamma}^{d} + O(h^2) = - 1 + O(h^2), \quad i = d-1,
      \\
      n_{\Gamma}^{d-1} + O(h^2) \quad i = d.
    \end{cases}
  \end{align}
Now adding~\eqref{eq:discrete-conormal-asymptotics}
and~\eqref{eq:contnuous-conormal-asymptotics} 
we arrive at the desired estimate 
by observing that 
$n_{\Gamma}^{d-1} - n_{\Gamma}^{d-1} n_{\Gamma}^{d}
=  n_{\Gamma}^{d-1}( 1 - n_{\Gamma}^{d})
=  n_{\Gamma}^{d-1} \cdot O(h^2)$.
\end{proof}
The previous lemma roughly states that push-forwarding a properly
scaled discrete co-normal vector results in a good approximation of
the co-normal vector field along the lifted edge.
As a result we can prove the final lemma of this section.
\begin{lemma} For $v \in H^1(\Gamma)^e + V_h$ 
  it holds
  \label{lem:energy-norm-equivalences}
  \begin{equation}
    \tn v^l \tn_{\ast}
    \lesssim  
    \tn v \tn_{\ast, h}
  \label{eq:energy-norm-equivalences}
  \end{equation}
\end{lemma}
\begin{proof} We only sketch the proof. Recalling the norm
  definitions~\eqref{eq:discrete-energy-norm} and
  \eqref{eq:discrete-energy-norm-lifted} and the norm
  equivalences from Lemma~\ref{lem:energy-norm-equivalences},
  it remains to show that 
  \begin{align}
    \| \mean{\nablas v^l} \|_{E^l} \sim \| \mean{\nabla v} \|_E
    \label{eq:cornormal-flux-equivalences}
  \end{align}
   where the mean incorporates the co-normal of the lifted edge $E^l$ and the
  discrete edge $E$, respectively. Lifting the $\mean{\nablash v}$ to $\Gamma$,
  equivalence~\eqref{eq:cornormal-flux-equivalences} can easily proven
  as in \eqref{eq:conormal-flux-lifted-estimate-I}--\eqref{eq:conormal-flux-lifted-estimate-III}
  by using the previous lemma to bound the deviation of the lifted discrete co-normal flux from
  the co-normal flux of the lifted edges.
\end{proof}

\section{Approximation Properties}
\label{sec:approximation-props}
Before we turn to the stability and a priori analysis in the next two
sections, we here
collect some standard inequalities 
and establish the appropriate interpolation estimates
which will be frequently used throughout the
remaining work. 

\subsection{Useful Inequalities}
First, we recall the following trace inequality for $v \in H^1(\mcT_h)$
\begin{equation}
  \| v \|_{\partial T} 
  \lesssim
  h^{-1/2} \|v \|_{T} +
  h^{1/2} \|\nabla v\|_{T}
  \quad \foralls T \in \mcT_h
  \label{eq:trace-inequality}
\end{equation}
If the intersection $\Gamma \cap T$ does not coincide with a boundary face of the mesh 
then the corresponding inequality
\begin{align}
  \| v \|_{\Gamma \cap T} 
  \lesssim
  h^{-1/2} \| v \|_{T} 
  + h^{1/2} \|\nabla v\|_{T} 
  \quad \foralls T \in \mcT_h
  \label{eq:trace-inequality-cut-faces}
\end{align}
holds whenever the surface $\Gamma$ is reasonably resolved by the
background mesh, i.e. the mesh size is chosen in accordance with the
local curvature, see \citet{HansboHansboLarson2003} for a proof.
Correspondingly, since the skeleton $\mcF_h$ consists
of shape-regular faces, we have
\begin{align}
  \| v \|_{E \cap F} 
  \lesssim
  h^{-1/2} \| v \|_{F} 
  + h^{1/2} \|\nabla v\|_{F}
  \quad \foralls E \in \mcE_h, \;
  \foralls F \in \mcF_h
  \label{eq:trace-inequality-cut-edges}
\end{align}
and if we iterate using~\ref{eq:trace-inequality} we see that for $v \in H^2(\mcT_h)$
\begin{align}
  \| v \|_{E \cap F} 
  \lesssim
  h^{-1} \| v \|_{T} 
  + \|\nabla v\|_{T}
  + h \|\nabla \otimes \nabla v\|_{T}
  \quad \foralls E \in \mcE_h, \;
  \foralls F \in \mcF_h
  \label{eq:trace-inequality-cut-edges-II}
\end{align}
In the following, we will also need the some
well-known inverse estimates for $v_h \in V_h$:
\begin{gather}
  \label{eq:inverse-estimate-grad}
  \| \nabla v_h\|_{T} 
  \lesssim
  h^{-1} 
  \| v_h\|_{T} \quad \foralls T \in \mcT_h
  \\
  \| v_h\|_{\partial T} 
  \lesssim
  h^{-1/2} 
  \| v_h\|_{T},
  \qquad 
  \| \nabla v_h\|_{\partial T} 
  \lesssim
  h^{-1/2} 
  \| \nabla v_h\|_{T} \quad \foralls T \in \mcT_h
  \label{eq:inverse-estimates-boundary}
\end{gather}
and the following ``cut versions'' 
(note that $K \cap T \not \subseteq \partial T$, $E \cap F \not
\subseteq \partial F$):
\begin{alignat}{5}
  \|v_h \|_{K \cap T} 
  &\lesssim
  h^{-1/2} \|v_h\|_{T},
  & & \qquad 
  \| \nabla v_h \|_{K \cap T} 
  &\lesssim
  h^{-1/2} \|\nabla v_h\|_{T}
  & &\quad \foralls K \in \mcK_h, \;
  \foralls T \in \mcT_h
  \label{eq:inverse-estimate-cut-v-on-K}
  \\
  \|v_h \|_{E \cap F}
  &\lesssim
  h^{-1/2} \|v_h\|_{F},
  & &\qquad
  \| \nabla v_h \|_{E \cap F} 
  &\lesssim
  h^{-1/2} \|\nabla v_h\|_{F}
  & &\quad \foralls E \in \mcE_h, \;
  \foralls F \in \mcF_h
  \label{eq:inverse-estimate-cut-on-E}
\end{alignat}
which are an immediate consequence of similar inverse estimates
presented in~\citet{HansboHansboLarson2003}.

\subsection{Interpolation Operators}
Next, we recall from~\citet{ScottZhang1990} that for
$v \in H^m(\mcT_h)$,
the standard Scott-Zhang interpolant $I_h:L^2(N_h) \rightarrow V_h$ satisfies the
local interpolation estimates
\begin{alignat}{3}
\| v - I_h v \|_{k,T} 
& \lesssim
  h^{l-k}| v |_{l,\omega(T)},
  & &\quad 0\leqslant k \leqslant l \leqslant \min\{2,m\} \quad &\foralls T\in \mcT_h
  \label{eq:interpest0}
  \\
\| v - I_h v \|_{l,F} &\lesssim h^{l-k-1/2}| v |_{l,\omega(F)},
  & &\quad 0\leqslant k \leqslant l - 1/2 \leqslant \min\{2,m\} - 1/2  \quad &\foralls F\in
  \mcF_{h} 
  \label{eq:interpest1}
\end{alignat}
where $\omega(T)$ denotes the patch of all elements sharing a
vertex with $T$. The patch $\omega(F)$ is defined analoguously.
Before we introduce a suitable interpolation operator $I_h:H^m(\Gamma) \to V_h$, we note that
by the coarea-formula (cf. \citet{EvansGariepy1992})
\begin{align*}
\int_{U_{\delta}} f(x) \,dx = \int_{-\delta}^{\delta} 
\left(\int_{\Gamma(r)} f(y,r) \, \mathrm{d} \Gamma_r(y)\right)\,\mathrm{d}r
\end{align*}
the extension operator $v^e$ defines a bounded operator
$H^m(\Gamma) \ni v \mapsto v^e \in H^m(U_{\delta}\Gamma))$
satisfying the stability estimate
\begin{align}
  \| v^e \|_{k,U_{\delta}(\Gamma)} \lesssim \delta^{1/2} \| v
  \|_{k,\Gamma}, \qquad 0 \leqslant k \leqslant m
\label{eq:stability-estimate-for-extension}
\end{align}
for $0 < \delta \leqslant \delta_0$ where the hidden constant depends only on the curvature of $\Gamma$.
With the help of the extension operator,
we construct an interpolation operator $I_h: H^m(\Gamma) \to V_h$ by
setting $I_h v = I_h v^e$ where we the liberty of using the same symbol.
Choosing $\delta_0 \sim h$,
it follows directly from combining the interpolation
estimate~\eqref{eq:interpest1} and the stability
estimate~\eqref{eq:stability-estimate-for-extension} that the following lemma holds:
\begin{lemma}
\label{lem:interpolenergy-lemma}
Setting $e_I = v^e - I_h v^e$ then for $v \in H^2(\Gamma)$, it holds that
\begin{align}
\label{eq:interpolenergy-lemma}
\sum_{F \in \mcF} 
\bigl(
h^{-2} \| e_I \|_{F}^2 
+ \| \nabla e_I \|_{F}^2
+ h^ 2 \| \nabla \otimes \nabla e_I \|_{F}^2
\bigr)
\lesssim  h^2 \| v \|_{2,\Gamma}^2
\end{align}
\end{lemma}
The interpolation operator satisfies the following interpolation error estimate
with respect to the discrete energy norm
\begin{lemma}
\label{lem:interpolenergy}
For $v \in H^2(\Gamma)$, it holds that
\begin{align}
\label{eq:interpolenergy}
\tn v^e - I_h v^e \tn_{\ast,h} + \tn v^e - I_h v^e \tn_{\mcF_h} &\lesssim  h \| v \|_{2,\Gamma}
\end{align}
\end{lemma}
\begin{proof}
By definition
\begin{align}
 \tn v^e - I_h v^e \tn_{\ast,h}^2
& = (\nablash (v^e - I_h v^e), \nablash (v^e - I_h v^e))_{\mcK_h}
 \nonumber
\\
& \quad + h \| \nablash (v^e - I_h v^e) \cdot n_{E,h} \|_{\mcE_h}^2
+ h^{-1} \| v^e - I_h v^e \|_{\mcE_h}
\\
&= I + II + III
\end{align}
which we estimate next.
\\
\noindent{\bf Term I.} Successively employing the trace
inequality~\eqref{eq:trace-inequality}, the interpolation
estimate~\eqref{eq:interpest0} and the stability
bound~\eqref{eq:stability-estimate-for-extension} with $\delta_0 \sim
h$,
it follows that
\begin{align}
I 
& =
\sum_{T \in \mcT_h} (\nablash (v^e - I_h v^e), \nablash (v^e - I_h v^e))_{\Gamma_h \cap T}
\\
&\lesssim
\sum_{T \in \mcT_h} \| \nabla (v^e - I_h v^e) \|^2_{T \cap \Gamma_h}
\\
&\lesssim
\sum_{T \in \mcT_h} 
\bigl(
h^{-1} \| \nabla (v^e - I_h v^e) \|^2_{T}
+ h \| \nabla \otimes \nabla (v^e - I_h v^e) \|^2_{T}
\bigr)
\\
&\lesssim
\sum_{T \in \mcT_h} 
h \| v^e \|^2_{2, \omega(T)}
\\
&\lesssim
h \| v^e \|^2_{2, U_{\delta_0}(\Gamma)}
\\
&\lesssim
h^2 \| v \|^2_{2, \Gamma}
\end{align}
{\bf Term II.} To estimate the second contribution, we employ the
trace inequaltiy~\eqref{eq:trace-inequality-cut-edges} to pass from
$E$ to $F$ which in combination with~\eqref{eq:interpest1}
and~\eqref{eq:stability-estimate-for-extension} gives
\begin{align}
II 
& \lesssim 
\sum_{E \in \mcE_h} h || \nablash  (v^e - I_h v^e) \|_{E}^2
\\
& \lesssim 
\sum_{E \in \mcE_h} 
\bigl(
|| \nabla  (v^e - I_h v^e) \|_{F}^2
+ h^2 || \nabla \otimes \nabla (v^e - I_h v^e) \|_{F}^2
\bigr)
\\
& \lesssim 
h^2 \| v \|^2_{2, \Gamma}
\end{align}
{\bf Term III.} Applying the same reasoning as for the estimate of
Term II, we obtain
\begin{align}
  III 
 \lesssim 
\sum_{E \in \mcE_h } h^{-1}\| e_I \|_{E}
 \lesssim
\sum_{F \in \mcF_h } 
\big(
 h^{-2}\| e_I \|_{F}^2
+ \| \nabla e_I \|_{F}^2
\bigr)
\lesssim
h^2 \| v \|^2_{2, \Gamma}
\end{align}
To estimate the remaining term $\tn e_I \tn_{\mcF_h}$, we simply apply
the previous Lemma~\ref{lem:interpolenergy-lemma}
to arrive at the desired bound.
\end{proof}

We conclude this section by briefly reviewing the
the Oswald interpolator, cf.~\citet{BurmanFernandezHansbo2006},
which will be of particular use in establishing discrete Poincar\'e-type
estimates in Section~\ref{ssec:poincare-estimates}. 
By construction, 
$\mcO_h: \mcP_{dc}^k(\mcT) \to \mcP_{c}^k(\mcT_h)$,
with $\mcP_{dc}^k(\mcT_h)$ and $\mcP_{c}^k(\mcT_h)$ denoting the space of
discontinuous and continuous piecewise polynonials of order $k$.
Therefore the Oswald interpolator can be thought of as a certain
``smoothing operator''
and the following lemma shows  
how the fluctuation $Id - \mcO_h$ can be controlled
in terms of jumps on the faces.
\begin{lemma}
  \label{lem:cip-fluctuation-control}
  For $v \in \mcP_{dc}^k(\mcT_h)$, it holds that
  \begin{align}
    \| \mcO_h(v) \|_T 
    &\lesssim \| v \|_{\omega(T)}
    \label{eq:oswald-stability}
    \\
    \label{eq:cip-fluctuation-control}
    \| v - \mcO_h(v) \|_{T}^2
    &\lesssim
    \sum_{F \in \mcF_h(T)}
    h \| \jump{v} \|_F^2
  \end{align}
where $\mcF_h(T)$ denotes the set of all faces $F \in \mcF_h$ with $F \cap T \neq \emptyset$.
\end{lemma}
A proof can be found in e.g.~\citet{BurmanFernandezHansbo2006},
together with the construction of the Oswald interpolator and a summary of some of its important
properties.

\section{Stability Estimates}
\label{sec:stability-estimates}
In this section, we demonstrate that the bilinear form
$a_h(\cdot,\cdot) + j_h(\cdot,\cdot)$ defined in
Section~\ref{ssec:dg-method} is both stable and coercive with
respect to its associated energy norm.
The proof of the coercivity result is based on a special discrete
Poincar\'e inequality which shows that for discrete functions, various $L^2$ norms
can be controlled by merely the tangential gradient and the semi-norm induced by stabilization form
$j_h(\cdot, \cdot)$.
As a major instrument in establishing such Poincar\'e inequalities,
we start by reviewing certain geometric covering relations of the
active background mesh and a certain stabilization mechanism provided by the
bilinear form $j_h(\cdot,\cdot)$.

\subsection{Fat Intersection Covering and Ghost Penalties}
\label{ssec:fat-intersection-covering}
Since the surface geometry is represented on fixed background mesh, the active background mesh $\mcT_h$
might contain elements which barely intersects the discretized surface $\Gamma_h$.
Such ``small cut elements'' typically prohibit the application of a whole set of
well-known estimates, such as interpolation estimates and inverse inequalities,
which typically rely on certain scaling properties.
As a partial replacement for the lost scaling properties we here recall from~\citet{BurmanHansboLarson2015} 
the concept of \emph{fat intersection coverings} of $\mcT_h$.

In \citet{BurmanHansboLarson2015} it was proved that
the active background mesh fulfills a fat intersection property which
roughly states that for every element in $\mcT_h$ there is
close-by element which has a significant intersection with $\Gamma_h$.
More precisely, let $x$ be a point on $\Gamma$ and let 
  $B_{\delta}(x) = \{y\in \RR^d: |x-y| < \delta\}$ 
  and 
  $D_{\delta} = B_{\delta}(x) \cap \Gamma$. 
  We define the sets of elements
  \begin{align}
    \mcK_{\delta,x} 
    = \{ K \in \mcK_h : \overline{K}^l \cap D_{\delta}(x) \neq
    \emptyset \},
    \qquad
    \mcT_{\delta,x} 
    = \{ T \in \mcT_h : T \cap \Gamma_h \in \mcK_{\delta,x} \}
    \label{eq:fat-intersection-covering}
  \end{align}
 With $\delta \sim h$ we use the notation $\mcK_{h,x}$ and
$\mcT_h,x$. For each $\mcT_h$, $h \in (0,h_0]$ there is a set of
points $\mcX_h$ on $\Gamma$ such that $\{\mcK_{h,x}, x \in \mcX_h \}$
and $\{\mcT_{h,x}, x \in \mcX_h \}$ are coverings $\mcT_h$ and
$\mcK_h$ with the following properties:
\begin{itemize}
  \item The number of set containing a given point $y$ is uniformly
    bounded
    \begin{align}
      \# \{ x \in \mcX_h : y \in \mcT_{h,x}  \} \lesssim 1 \quad
      \foralls y \in \RR^d
    \end{align}
  for all $h \in (0,h_0]$ with $h_0$ small enough.
\item
  The number of elements in the sets $\mcT_{h,x}$ is uniformly bounded
  \begin{align}
    \# \mcT_{h,x} \lesssim 1 
    \quad \foralls x \in \mcX_h
  \end{align}
  for all $h \in (0,h_0]$ with $h_0$ small enough, and each element in
  $\mcT_{h,x}$ shares at least one face with another element in
  $\mcT_{h,x}$.
\item $\foralls h \in (0,h_0]$ with $h_0$ small enough, and $\foralls x \in \mcX_h$, $\exists
  T_x \in \mcT_{h,x}$ that has a large (fat) intersection with
  $\Gamma_h$ in the sense that
  \begin{align}
  | T_x | \sim h | T_x \cap \Gamma_h | = h | K_x | 
  \quad \foralls x \in \mcX_h
  \end{align}
\end{itemize}

To make use of the fat intersection property, we will need the following lemma
which describes how to the control of discrete 
functions on potentially small cut elements can be transferred to
their close-by neighbors with large intersection by using the face-based
stabilization terms appearing in $j_h(\cdot,\cdot)$.
\begin{lemma}
  \label{lem:l2norm-control-via-jumps}
  Let $v$ be a discontinuous, piecewise linear function defined on a quasi-uniform mesh $\mcT_h$
  consisting of element $\{T\}$ and consider 
  a the macro-element $\mathcal{M} = T_1 \cup T_2$ formed by
  any two elements $T_1$ and $T_2$ sharing a face $F$. 
  Then
  \begin{align}
    \| v \|_{T_1}^2 &\lesssim  \|v\|_{T_2}^2
    + h \| \jump{u} \|_F^2 
    + h^3 \| n_F \cdot \jump{\nabla v} \|_F^2 
    \label{eq:l2norm-control-via-jumps-I}
    \\
    \| \nabla v \|_{T_1}^2 
    &\lesssim  \|\nabla v\|_{T_2}^2
    + h^{-1} \| \jump{v} \|_F^2 
    + h  \|n_F \cdot \jump{\nabla v} \|_F^2 
    \label{eq:l2norm-control-via-jumps-II}
  \end{align}
  with the hidden constant only depending on the quasi-uniformness parameter.
\end{lemma}
\begin{proof}
A proof of the first estimate can be found in~\citet{MassingLarsonLoggEtAl2013}.
Setting $v = \nabla w$ in~\eqref{eq:l2norm-control-via-jumps-I},
the second inequality follows directly from the first one once the following face
base inverse estimate is established:
\begin{align}
  \|\jump{\nabla v}\|_F^2 
  \lesssim 
  \| n_F \cdot \jump{\nabla v}  \|_F^2
  + h^{-2} \| \jump{v} \|_F^2
  \label{eq:jump-full-gradient-split}
\end{align}

To prove~\eqref{eq:jump-full-gradient-split}, rewrite
$\nabla v = n_F \cdot \nabla v + P_F \nabla v$ 
with 
$P_F = Id - n_F \otimes n_F$ 
and observe that the 
face tangential part of the gradient satisfies
$\| \jump{P_F \nabla u } \|_F^2 \lesssim h^{-2} \| \jump{u}\|_F^2$.
\end{proof}

\subsection{Discrete Poincar\'e Estimates}
\label{ssec:poincare-estimates}
Next, we derive a discrete Poincar\'e estimate showing that for $v \in V_h$
the $L^2$ norm on the active background mesh can be controlled by the tangential surface gradient
and the semi-norm induced by $j_h(\cdot,\cdot)$.
While the final Poincar\'e estimate bears resemblance to estimates
presented in \citet{Brenner2003}, there are two major
differences. First, the active background mesh and thus the domain
$N_h$ varies with the mesh size, and second, the full $\RR^d$ gradient
of $v \in V_h$ is not available in our surface method.
We start with the following lemma.
\begin{lemma}
  \label{lem:poincare-I} For $v \in V_h$, the following estimate holds 
  \begin{equation}
    h^{-1}\| v - \lambda_{\Gamma_h}(v) \|^2_{N_h}
    \lesssim
    \| \nablash v \|_{\Gamma_h}^2 
    + h \| \nabla v \|_{N_h}^2 
    + h^{-2} \| \jump{v} \|^2_{\mcF_h}
  \label{eq:poincare-I}
\end{equation}
  for $0<h \leq h_0$ with $h_0$ small enough.
\end{lemma}
\begin{proof}
  Without loss of generality we can assume that $\lambda_{\Gamma_h}(v) = 0$.
  Apply~\eqref{eq:l2norm-control-via-jumps-I} and
  \eqref{eq:inverse-estimates-boundary} to obtain
  \begin{align}
    \| v \|_{N_h}^2 
    &\lesssim \sum_{x \in \mcX_h} \| v \|_{\mcT_{h,x}}^2
    \\
    &\lesssim
    \sum_{x \in \mcX_h} \| v \|_{T_x}^2 
    + h \| \jump{v} \|_{\mcF_h}^2
    + h^3 \| n_F \cdot \jump{\nabla v} \|_{\mcF_h}^2
    \\
    &\lesssim
    \sum_{x \in \mcX_h} \| v \|_{T_x}^2 
    + h \| \jump{v} \|_{\mcF_h}^2
    + h^2 \| \nabla v \|_{N_h}^2
    \label{eq:poincare-proof-I}
\end{align}
Thus it is sufficient to estimate the first term
in~\eqref{eq:poincare-proof-I}.
For $v \in V_h$, we define a piecewise constant version satisfying
$\overline{v}|_T = \tfrac{1}{|T|} \int_T v \dx$. Clearly 
$\|v - \overline{v}\|_T \lesssim h \| \nabla v \|_T$.
Adding and subtracting $\overline{v}$ gives
\begin{align}
  \sum_{x \in \mcX_h} \| v \|_{T_x}^2
  &\lesssim
  \sum_{x \in \mcX_h}  \| v - \overline{v} \|_{T_x}^2
  +
  \sum_{x \in \mcX_h}  \| \overline{v} \|_{T_x}^2
  \\
  &\lesssim
  h^2 \| \nabla v \|_{N_h}^2
  + \sum_{x \in \mcX_h} h \| \overline{v} \|_{K_x}^2
  \\
  &\lesssim
  h^2 \| \nabla v \|_{N_h}^2
  + h \| v \|_{\Gamma_h}^2
  + h \| v - \overline{v} \|_{\Gamma_h}^2
  \\
  &\lesssim
  h^2 \| \nabla v \|_{N_h}^2
  + h \| v \|_{\Gamma_h}^2
  \label{eq:transition-to-surface-I}
\end{align}
It remains to estimate the last term in~\eqref{eq:transition-to-surface-I}.
To apply a Poincar\'e estimate on the surface~$\Gamma$,
we introduce a smoothed version $\widetilde{v} = \mcO_h(v)
\in H^1(\Gamma)$ of $v$. Then
\begin{align}
  h \| v \|_{\Gamma_h}^2
  &\lesssim
  h \| \widetilde{v} \|_{\Gamma_h}^2
  +
  h \| v - \widetilde{v} \|_{\Gamma_h}^2
  \lesssim
  h \| \widetilde{v} \|_{\Gamma_h}^2
  +
  \| v - \widetilde{v} \|_{N_h}^2
  \lesssim
  h \| \widetilde{v}^l \|_{\Gamma}^2
  +
  h \| \jump{v} \|_{\mcF_h}^2
\end{align}
where we used the inverse
estimate~\eqref{eq:inverse-estimate-cut-v-on-K} to pass from
$\Gamma_h$ to $N_h$ and the fluctation
control~\eqref{eq:cip-fluctuation-control} for the Oswald
interpolator.
Now applying a standard Poincar\'e estimate to the first term yields
  \begin{align}
    h \| \widetilde{v}^l \|_{\Gamma}^2
    \lesssim h \lambda_{\Gamma}(\widetilde{v}^l)^2 + h \| \nablas
    \widetilde{v}^l
    \|_{\Gamma}^2 = I + II.
  \end{align}
  which we estimate next.
  \\
  {\bf Term I.}
  Since $\lambda_{\Gamma_h}(v) = 0$, the first term can be considered
  as the error of the mean value which can be bounded by
  \begin{align}
    I 
    &= h 
    \left(
      \lambda_{\Gamma}(\widetilde{v}^l)
      - \lambda_{\Gamma_h}({v})
    \right)^2
    \lesssim
    \dfrac{h}{|\Gamma_h|}
    \left(
   \int_{\Gamma_h} | \widetilde{v} - v|^2 \ds
   + \int_{\Gamma_h} ( 1 - c)^2 \widetilde{v}^2 \ds
    \right)
\label{eq:error-of-average-est}
  \end{align}
  with $c = |\Gamma_h||\Gamma|^{-1} |B|$.
  We note  that 
  $\| 1 - c \|^2_{L^\infty(\Gamma)} \lesssim h^4$
  thanks to~\eqref{eq:detBbound}. Then
  \begin{align}
    I &\lesssim 
    h \| \widetilde{v} - v \|_{\Gamma_h}^2
    + h^5 \| \widetilde{v} \|_{\Gamma_h}^2
    \lesssim
    \| \widetilde{v} - v \|_{N_h}^2 
    + h^4 \| \widetilde{v} \|_{N_h}^2 
    \lesssim h
    \|\jump{v}\|_{\mcF_h}^2
    + h^4 \| v \|_{N_h}^2 
  \end{align}
  where the estimates
  ~\eqref{eq:cip-fluctuation-control} and
  \eqref{eq:oswald-stability} for the Oswald interpolator were used in the last
  step.
  \\
  {\bf Term II.} To estimate the second term, apply the inverse
  estimates~\eqref{eq:inverse-estimate-cut-v-on-K} and
  ~\eqref{eq:inverse-estimate-grad} to obtain 
  \begin{align}
    II 
     \lesssim 
    h \| \nablash \widetilde{v} \|_{\Gamma_h}^2
    &\lesssim
    h \| \nablash v \|_{\Gamma_h}^2
    + h \| \nablash (v - \widetilde{v}) \|_{\Gamma_h}^2
    \\
    &\lesssim
    h \| \nablash v \|_{\Gamma_h}^2
    +
    \| \nabla(\widetilde{v}  - v )\|_{N_h}^2
    \\
    &\lesssim
    h \| \nablash v \|_{\Gamma_h}^2
    +
    h^{-2}\| \widetilde{v}  - v \|_{N_h}^2
    \\
    &\lesssim
    h \| \nablash v \|_{\Gamma_h}^2
    +
    h^{-1}\| \jump{v}\|_{\mcF_h}^2
  \end{align}
  Collecting all terms and dividing by $h$, we see that
  \begin{align}
    h^{-1}\| v \|_{N_h}^2 
    \lesssim 
    \| \nablash v \|_{\Gamma_h}^2 
    + h \| \nabla v \|_{N_h}^2 
    + h^{-2} \|\jump{v}\|_{\mcF_h}^2
    + h^3 \| v \|_{N_h}^2
  \end{align}
  which gives the desired estimates since the last term 
  can be absorbed into the left-hand side for $h$ small enough.
\end{proof}

The next lemma describes how
the full gradient $\nabla v$ on $N_h$ can be eliminated
from~\eqref{eq:poincare-I}.
The main idea is to apply the previous lemma to $\nabla v$
to show that the $h^{1/2}$-scaled, full gradient 
can be controlled in terms of the tangential gradient 
and the face stabilization~\eqref{eq:jh-def}.
\begin{lemma}
  For $v \in V_h$, the following estimate holds
  \begin{align}
    h \| \nabla v \|_{N_h}^2
    \lesssim
    h^2 \| \nablash v \|_{\Gamma_h}^2
    + 
    \tn v \tn_{\mcF_h}^2
    \label{eq:control-entire-grad-by-tang-grad}
  \end{align}
  \label{lem:control-entire-grad-by-tang-grad}
\end{lemma}
An immediate consequence is 
\begin{corollary}
  \label{cor:discrete-poincare}
  Let $h \in (0,h_0]$ with $h_0$ small
  enough. Then the following estimate holds:
  \begin{equation}
    h^{-1}\| v - \lambda_{\Gamma_h}(v) \|^2_{N_h}
    \lesssim
    \| \nablash v \|^2_{\Gamma_h}
    + \tn v \tn^2_{\mcF_h}
    \quad \forall v \in V_h.
  \end{equation}
\end{corollary}
As a result, $\tn \cdot \tn_h$ defines a norm on
$V_{h,0}$ which satisfies the discrete Poincar\'e estimates
\begin{align}
  h^{-1}\|v\|^2_{N_h}
  \lesssim \tn v \tn^2_h
  \label{eq:discrete-poincare-Nh}
  \\
  \|v\|^2_{\Gamma_h}
  \lesssim \tn v \tn^2_h
  \label{eq:discrete-poincare-Gammah}
\end{align}
where the second version is obtained from the first using the inverse
estimate~\eqref{eq:inverse-estimate-cut-on-E}.
\begin{proof}(Lemma~\ref{lem:control-entire-grad-by-tang-grad})
  Choosing $a = \lambda_{\Gamma_h}(\nabla v)$,
  the previous lemma gives
  \begin{align}
   h  \| \nabla v \|_{\mcT_h}^2
   &\lesssim
   h \| a \|_{\mcT_h}^2
   + 
   h  \| \nabla v - a \|_{\mcT_h}^2
   \\
   &\lesssim 
   h \| a \|_{\mcT_h}^2
   +
   \| [ \nabla u ] \|_{\mcF_h}^2
   \label{eq:constant-split-of}
  \end{align}
It remains to estimate the first term 
$ h \| a \|_{\mcT_h}^2$. Clearly, 
\begin{align}
h \| a \|_{\mcT_h}^2 \lesssim  h^2 \| a \|_{\Gamma_h}^2
  \label{eq:kick-back-step-1}
\end{align}
and since 
$\| \Ps a \|_{\Gamma_h} \lesssim \| a \|_{\Gamma}$ 
by a
finite dimensionality argument,
we obtain
\begin{align}
  \| a \|_{\Gamma_h} 
  \lesssim
  \| a \|_{\Gamma} 
  \lesssim
  \| \Ps a \|_{\Gamma} 
  \lesssim
  \| \Ps a \|_{\Gamma_h} 
  & \lesssim
  \| \Psh a \|_{\Gamma_h} 
  + \| (\Psh - \Ps) a \|_{\Gamma_h} 
  \\
  &\lesssim
  \| \Psh a \|_{\Gamma_h} 
  + h \| a \|_{\Gamma_h} 
\end{align}
We can kick-back $h \| a \|_{\Gamma_h}$
and absorb it on the left-hand side whenever $h \leqslant h_0$
for some $h_0$ small enough. Consequently, 
\begin{align}
h^2 \| a \|_{\Gamma_h} \lesssim 
h^2 \| \Psh a \|_{\Gamma_h}^2
  \label{eq:kick-back-step-2}
\end{align}
Now using the boundedness of $\Psh$ and the inverse
inequality~\eqref{eq:inverse-estimate-cut-v-on-K}, it follows that
\begin{align}
h^2 \| \Psh a \|_{\Gamma_h}^2
&\lesssim
h^2 \| \Psh \nabla v \|_{\Gamma_h}^2
+
h^2 \| \Psh (a - \nabla v) \|_{\Gamma_h}^2
\nonumber
\\
& \lesssim
h^2 \| \nablash  v \|_{\Gamma_h}^2
+
h \| a - \nabla v \|_{\mcT_h}^2
\nonumber
\\
& \lesssim
h^2 \| \nablash  v \|_{\Gamma_h}^2
+
   \| [ \nabla v ] \|_{\mcF_h}^2
  \label{eq:kick-back-step-3}
\end{align}
Collecting~\eqref{eq:kick-back-step-1}, \eqref{eq:kick-back-step-2}
and~\eqref{eq:kick-back-step-3} and
using~\eqref{eq:jump-full-gradient-split}, we conclude that
\begin{align}
  h \| a \|_{\mcT_h}^2 
  &\lesssim  
  h^2 \| \nablash  v \|_{\Gamma_h}^2
  +
  \| [ \nabla u ] \|_{\mcF_h}^2
  \\
  & \lesssim
  h^2 \| \nablash  v \|_{\Gamma_h}^2
  +
  \| [ n_F \cdot \nabla u ] \|_{\mcF_h}^2
  +
  h^{-2} \| [ u ] \|_{\mcF_h}^2
\end{align}
which in combination with ~\eqref{eq:constant-split-of} yields the desired inequality.
\end{proof}

\subsection{Coercivity and Continuity}
\label{ssection:coercivity}
\begin{lemma} The following estimate holds
\begin{equation}
  h \| n_E \cdot \nablash v \|^2_{\mcE_h}
\lesssim
\| \nablash v \|^2_{\Gamma_h}
+\tn v \tn_{\mcF_h}^2
\label{eq:inverse-estimate-for-conormal-flux}
\end{equation}
for $0< h \leq h_0$ with $h_0$ small enough.
\label{lem:inverse-estimate-for-conormal-flux}
\end{lemma}
\begin{proof}

We start with observing that since both $n_{\Gamma_h}$ and $\nabla v$
are piecewise constant functions on $\mcT_h$ so is $\nablash v$.
Then successively employing the inverse estimates~\eqref{eq:inverse-estimate-cut-on-E},\eqref{eq:inverse-estimates-boundary} 
  \begin{align}
  h \| n_E \cdot \nablash v \|^2_{\mcE_h}
  & \lesssim 
  h^{-1} \| \nablash v \|^2_{\mcT_h} 
  \\
  &\lesssim 
  h^{-1} 
  \sum_{x \in \mcX_h}
  \| \nablash v \|^2_{\mcT_{x,h}} 
  \\
  &\lesssim 
  h^{-1} 
  \sum_{x \in \mcX_h}
  \| \nablash v \|^2_{T_x}
   + \|\jump{\nablash v} \|_{\mcF_h}^2
  \\
  &\lesssim 
  \sum_{x \in \mcX_h}
  \| \nablash v \|^2_{K_x}
  + \|\jump{\nablash v} \|_{\mcF_h}^2
\label{eq:conormal-flux-est-step-I}
  \end{align}
where in the last to steps we applied
Lemma~\ref{lem:l2norm-control-via-jumps},
Eq.~\ref{eq:l2norm-control-via-jumps-I} to $\nablash$ and the fat
intersection covering property as defined in
Section~\ref{ssec:fat-intersection-covering}.
Since the first term in~\eqref{eq:conormal-flux-est-step-I} is bounded by 
$\| \nablash v \|_{\Gamma_h}^2$, it only remains to estimate the second term.
Denoting with 
$\dgmean{v}|_F = \tfrac{1}{2}(v^+_T + v^-_T)$
the mean value of any piecewise smooth function or operator accross
$F$, standard identities for the jump-operator 
and the boundedness of $\Psh$ give
\begin{align}
  \| \jump{\nablash v} \|_{\mcF_h}^2
  \lesssim
  \|  \jump{\Psh} \dgmean{\nabla v} \|_{\mcF_h}^2
  + \|  \dgmean{\Psh} \jump{\nabla v} \|_{\mcF_h}^2
  \lesssim
  \|  \jump{\Psh} \dgmean{\nabla v} \|_{\mcF_h}^2
  + \| \jump{\nabla v} \|_{\mcF_h}^2
  = I + II
\end{align}
To conclude the proof we need to estimate $\jump{\Psh}$.
But as in the proof of Lemma~\ref{lem:BBTbound}, we see that
\begin{align}
  \|\jump{\Psh}\|_{L^{\infty}(F)}^2
  = 
  \|\jump{\Psh - \Ps}\|_{L^{\infty}(F)}^2
  \lesssim h^2
\end{align}
and hence by Lemma~\ref{lem:control-entire-grad-by-tang-grad} we
arrive at the final estimate
\begin{align}
  I \lesssim
  \sum_{T\in\mcT_h} h^2 \|\nabla v \|_{\partial T}^2
  \lesssim
  h \|\nabla v \|_{N_h}^2
  \lesssim
  h^2 \|\nablash v \|_{\Gamma_h}^2 + \tn v \tn_{\mcF_h}^2
\end{align}
\end{proof}

Before we turn to the main stability result, we quickly note
as an immediate consequence that
  \begin{align}
    \tn v \tn_{\ast,h} 
    \sim \tn v \tn_{h} \quad \foralls v \in V_h
    \label{eq:discrete-energy-norm-equivalence}
  \end{align}
\begin{proposition} 
  The discrete bilinear form $A_h(\cdot,\cdot)$ is both coercive and
  stable with respect to the discrete
  energy-norm~\eqref{eq:discrete-energy-norm}, that is it satisfies
\begin{equation}
  \label{eq:coercivity}
\tn v \tn_h^2 \lesssim A_h(v,v) \quad \forall v \in V_h
\end{equation}
and
\begin{equation}\label{eq:continuity}
 A_h(v,w) \lesssim \tn v \tn_h\, \tn w \tn_h \quad \forall v,w \in V_h
\end{equation}
whenever $\beta_E$, $\beta_F$ and $\gamma$ are chosen large enough.
\label{prop:stability-estimates-Ah}
\end{proposition}
\begin{proof} 
{\bf (\ref{eq:coercivity}):}
Starting from the definition of $A_h$ and applying Cauchy's inequality with $\epsilon$, we have
\begin{align}
A_h(v,v) &\gtrsim \| \nablash v\|^2_{\mcK_h}
 - \epsilon \|h^{1/2}\langle n_E \cdot \nabla v \rangle\|_{\mcE_h}^2
+ (\beta_E - \epsilon^{-1}) \|h^{-1/2}\jump{v}\|_{\mcE_h}^2
+ \tn v \tn_{\mcF_h}^2
\end{align}
and employing Lemma~\ref{lem:inverse-estimate-for-conormal-flux}, we immediately see
that
\begin{align}
\nonumber
A_h(v,v) &\gtrsim \| \nablash v\|^2_{\mcK_h}
+ \|h^{-1/2}\jump{v}\|_{\mcE_h}^2
+ \tn v \tn_{\mcF_h}^2
\end{align}
when choosing $\epsilon$ small and $\beta_E$ large enough.
\\
\noindent{\bf (\ref{eq:continuity}):} Follows directly from
the Cauchy-Schwarz inequality and the norm
equivalence~\eqref{eq:discrete-energy-norm-equivalence}. 
\end{proof}

\section{A Priori Error Estimates}
\label{sec:a-priori-estimates}
The goal of this section is to formulate and prove the main a priori
estimates for the proposed cut discontinuous Galerkin method.
We proceed in two steps. First, an abstract Strang-type lemma is
established which reveals that the overall error can be attributed to
two sources, namely an interpolation error and a quadrature error 
caused by the mismatch of the smooth surface and its discrete counterpart.
Then the quadrature error is bounded using the geometric estimates
from Section~\ref{sec:domain-perturbation}.

\subsection{Strang's Lemma}
Recalling the definition of the lifted discontinuous Galerkin
form~\eqref{eq:lifted-dg-form}, we can state 
\begin{lemma} 
  \label{lem:strang}
With $u$ the solution of \eqref{eq:LB}
and $u_h$ the solution of \eqref{eq:weak-cutfem-formulation} it holds
\begin{align}
  \tn u^e - u_h \tn_{\ast, h} &\leq \tn u^e - I_h u^e \tn_{\ast, h}
\\ \nonumber
&\qquad
+ \sup_{v \in V_h} \tn v \tn_h^{-1} 
\Big( 
a_h(I_h u^e,v) - a_h^l(I_h^l u^e,v^l)\Big)
\\ \nonumber
&\qquad
+ \sup_{v \in V_h} \tn v \tn_h^{-1} \Big( l(v^l) - l_h(v)\Big)
\end{align}
\end{lemma}
\begin{proof}
Let $e_h = I_h u^e - u_h$. Thanks to triangle inequality 
$\tn u^e - u_h \tn_{\ast,h} \leqslant \tn u^e - I_h u^e \tn_{\ast,h}
+ \tn e_h \tn_{\ast,h}$ and the equivalences of the discrete energy
norms~\eqref{eq:discrete-energy-norm-equivalence} for $v \in V_h$, it is enough to estimate $\tn e_h
\tn_{h}$.
We have
\begin{align}
 \tn e_h \tn_h^2 & \lesssim A_h(I_h u^e - u_h,e_h)
\\
&=A_h(I_h u^e,e_h) - l_h(e_h)
\\
&=A_h(I_h u^e,e_h) - a_h^l(u, e_h^l) + l(e_h^l)
- l_h(e_h)
\\
&=  \Big( a_h(I_h u^e,e_h) - a_h^l(I_h^l u^e, e_h^l)\Big) + \Big( l(e_h^l) - l_h(e_h)\Big)
\\ \nonumber
&\qquad - a_h^l(u - I_h^l u^e, e_h^l) + j_h( I_h u^e, e_h)
\\
&=I + II + III + IV
\end{align}
The bounds for first two terms are immediate:
\begin{align}
  I &\leq \tn e_h \tn_h \sup_{v \in V_h} \tn v \tn_h^{-1} \Big( a_h(I_h
  u^e,v) - a_h^l(I_h^l u,v^l)\Big)
\\
II &\leq \tn e_h \tn_h \sup_{v \in V_h} \tn v \tn_h^{-1} \Big( l(v^l) - l_h(v)\Big)
\end{align}
Thanks to~\eqref{eq:energy-norm-equivalences}
and~\eqref{eq:discrete-energy-norm-equivalence}, we also have
\begin{align}
  III &\leq \tn u - I_h^l u^e\tn_{\ast} \, \tn e_h^l \tn_{\ast}
  \lesssim \tn u^e - I_h u^e \tn_{\ast,h} \, \tn e_h \tn_{h}
  \\
  IV &= j_h(I_h u^e - u^e,e_h) \leq  \tn u^e - I_h u^e \tn_{\mcF_h}
  \tn e_h \tn_{\mcF_h}
\end{align}
Collecting the estimates and dividing by $\tn e_h \tn_h$ concludes the proof.
\end{proof}

\subsection{Quadrature Error Estimates}
Next, we show how the abstract quadrature error appearing in Strang's lemma can be bounded in terms of the mesh size and the discrete energy norm.
\begin{lemma}
\label{lem:quadrature-est} 
The following estimates hold
  \begin{align}
    \label{eq:quadrature-est-a}
    |a_h^l(v^l,w^l) - a_h(v,w)| &\lesssim h^2 \tn v \tn_h \tn w \tn_h
    \quad \forall v,w \in V_h
    \\
    \label{eq:quadrature-est-l}
    |l(v^l) - l_h(v)| &\lesssim h^2 \tn v \tn_h
    \quad \forall v \in V_h
  \end{align}
\end{lemma}
\begin{proof} {\bf (\ref{eq:quadrature-est-a}):}
  After lifting the bilinear~$a_h(\cdot, \cdot)$ from $\Gamma_h$ to
  $\Gamma$, each of its contribution can be estimated by successively
  employing the bounds for
  determinants~\eqref{eq:detBbound}--\eqref{eq:detBEbound},
  the operator norm estimates~\eqref{eq:BBTbound},
  and the norm
  equivalences~\eqref{eq:norm-equivalences-l2}--\eqref{eq:norm-equivalences-h1}.
  In doing so we obtain for each $K \in \mcK_h$
  \begin{align}
    \nonumber
    &(\nablas v^l,\nablas w^l)_{K^l} - (\nablash v,\nablash w)_{K}
    \\
    &\qquad =(\nablas v^l,\nabla w^l)_{K^l} - ((\nablash v)^l,(\nablash w)^l |B|_d^{-1})_{K^l}
    \\
    &\qquad =((\Ps - |B|_{d-1}^{-1} B B^T)\nablas v^l,\nablas w^l)_{K^l}
    \\
    &\qquad \lesssim h^2 \| \nablas v^l\|_{K^l} \|\nablas w^l\|_{K^l}
  \end{align}
  Keeping the estimate for the lifted discrete
  co-normal~\eqref{eq:cornormal-lifted-bound} in mind, each edge
  contribution gives
  \begin{align}
    \nonumber
    &(\langle n_E \cdot \nablas v^l \rangle, [w^l])_{E^l}
    - (\langle n_{h,E} \cdot \nablash v \rangle, [w])_{E}
    \\
    &\qquad = (\langle n_E \cdot \nablas v^l \rangle, [w^l])_{E^l}
    - (\langle (n_{h,E} \cdot  B^{T} \nablas v^l \rangle, [w^l])_{E}
    \label{eq:conormal-flux-lifted-estimate-I}
    \\
    &\qquad = (\langle (n_E - |B|_{d-2}^{-1} B n_{h,E})\cdot \nablas v^l \rangle, [w^l])_{E^l}
    \label{eq:conormal-flux-lifted-estimate-II}
    \\
    &\qquad \lesssim h^2 \| h^{1/2} \langle \nablas v^l \rangle \|_{E^l}
    \|h^{-1/2} [w^l]\|^2_{E^l}
    \label{eq:conormal-flux-lifted-estimate-III}
    \\
    &\qquad \lesssim h^2 \| h^{1/2} \langle \nablash v \rangle \|_{E}
    \|h^{-1/2} [w]\|^2_{E}
  \end{align}
  and similarly
  \begin{align}
    \nonumber
    &(h^{-1}[v^l], [w^l])_{E^l}
    - (h^{-1}[v], [w])_{E}
    \\
    &\qquad = (h^{-1}(1-|B|_{d-2}^{-1})\jump{v^l}, \jump{w^l})_{E^l}
    \\
    &\qquad  \lesssim h^2
    \|h^{-1/2} \jump{v}\|^2_{E}\|h^{-1/2} \jump{w}\|^2_{E}
  \end{align}
  \\
  \noindent  {\bf (\ref{eq:quadrature-est-l}):} For the right hand side we have
  \begin{align}
    l(w^l)- l_h(w) &= (f,w^l)_\Gamma - (f^e, w)_\Gammah
    \\
    &= (f,w^l (1 -|B|_d^{-1}))_\Gamma
    \\
    &\lesssim h^2 \|w^l\|_{\Gamma}
    \\
    &\lesssim h^2 \tn w \tn_{h}
  \end{align}
  where in last step, the Poincar\'e
  inequality~\eqref{eq:discrete-poincare-Gammah} was used
  after passing from $\Gamma$ to $\Gamma_h$.
\end{proof}

\subsection{Error Estimates}
\label{ssec:apriori-est}
Finally, we can state and prove the main a priori estimates.
\begin{theorem} The following a priori error estimates hold
  \label{thm:aprioriest}
  \begin{align}
    \label{eq:energyest}
    \tn u^e - u_h \tn_{\ast, h} &\lesssim h \| f \|_{\Gamma}
    \\ \label{eq:ltwoest}
    \| u^e - u_h \|_{\Gammah} &\lesssim h^2 \| f \|_{\Gamma}
  \end{align}
\end{theorem}
\begin{proof} {\bf (\ref{eq:energyest}):} Combining Strang's
  Lemma~\ref{lem:strang}, the interpolation error
  estimate~\eqref{eq:interpolenergy},
  the quadrature estimates in Lemma \ref{lem:quadrature-est}, and finally
  the elliptic regularity estimate \eqref{eq:ellreg} the proof
  follows immediately.

  \noindent  
  {\bf (\ref{eq:ltwoest}):} Since the average 
  $\lambda_{\Gamma}(u_h^l)= |\Gamma|^{-1}\int_\Gamma u_h^l$ 
  is not equal to zero we decompose the error as follows 
  \begin{align}
    e = u - u_h^l = \underbrace{u - (u_h^l - \lambda_\Gamma(u_h^l))}_{\tilde{e}} 
    + \underbrace{\lambda_\Gammah(u_h) -  \lambda_\Gamma(u_h^l)}_{e_c} 
  \end{align} 
  where the first term has average zero on $\Gamma$ and the 
  second term is the error in the average. 

  To estimate $\tilde{e}$ we let $\phi \in H^1(\Gamma)/ \mathbb{R}$ be
the solution to the dual problem $(\nablas v,\nablas \phi)_{\Gamma} =
(\tilde{e},v)_\Gamma \; \forall v \in H^1(\Gamma)/\mathbb{R}$. Then we
have the elliptic regularity estimate $\|\phi \|_{H^2(\Gamma)}
\lesssim \|\tilde{e}\|_{\Gamma}$.  Exploiting the fact that
$\jump{\phi}|_E^l = \jump{n_E^l \cdot \nablas \phi}_E^l = 0$ by regularity
and setting $v=\tilde{e}$ and we obtain
  \begin{align}
    \|\tilde{e}\|^2_\Gamma &= a_h^l(\tilde{e},\phi) 
    \\
    &=a_h^l(e,\phi)
    \\
    &= a_h^l(e,\phi - (I_h \phi)^l ) + a_h^l(e,(I_h \phi)^ f \\
    &\lesssim \tn e \tn_{\ast}\, \tn \phi - (I_h \phi^e)^l \tn_{\ast}
    + | a_h^l(e,(I_h \phi^e)^l |
    \\ \label{eq:L2a}
    &\lesssim h \tn e \tn_{\ast} \, \|\phi \|_{H^2(\Gamma)}
    + | a_h^l(e,(I_h \phi)^l |
    \\
    &\lesssim h^2 \| f \|_{\Gamma} \| \tilde{e} \|_{\Gamma}
    + | a_h^l(e,(I_h \phi)^l |
  \end{align}
  where we added and subtracted an interpolant and 
  estimated the first term using the interpolation estimate~\eqref{eq:interpolenergy}
  followed by the elliptic regularity estimate for the dual solution.
  To estimate the second term on the right hand side 
  we note that we have the identity
  \begin{align}
    a(e,v^l) &= a(u,v^l) - a(u_h^l,v) 
    \\
    &=l(v^l) - l_h(v) + a_h(u_h,v) - a(u_h^l,v^l) \qquad \forall v\in V_h
  \end{align}
  Thus using the quadrature estimates~\eqref{eq:quadrature-est-a}) and 
  \eqref{eq:quadrature-est-l}) we obtain
  \begin{align} 
    \label{eq:L2b}
    |a_h^l(e,(I_h \phi^e)^l)|&\lesssim 
    h^{2} \Big( \|f\|_\Gamma + \tn u_h\tn_h \Big) 
    \tn I_h \phi^e \tn_h
    \lesssim 
    h^{2} \|f\|_\Gamma \| e \|_\Gamma
  \end{align}
  where at last we used the stability $\tn u_h \tn_h \lesssim \|f^e\|_{\Gamma_h} \lesssim \|f\|_\Gamma$ of the method and the energy 
  norm stability of the interpolant together with energy stability of the 
  solution to the dual problem. Together estimates (\ref{eq:L2a}) and 
  (\ref{eq:L2b}) give the estimate 
  \begin{equation}
    \|\tilde{e} \|_\Gamma \lesssim h^2 \| f\|_{\Gamma}
  \end{equation}

  To conclude the proof, we note that error in the average $e_c$ can be exactly estimated
as the last term in~\eqref{eq:error-of-average-est}, yielding 
  \begin{align}
    \|e_c\| \lesssim  |\lambda_\Gamma( u_h^l) - \lambda_\Gammah(u_h)|
    \lesssim h^2 \| u_h \|_{\Gamma_h} \lesssim h^2 \| f \|_{\Gamma}
  \end{align}
\end{proof}

\section{Condition Number Estimate}
\label{sec:condition-number-estimate}
Next, we show that the condition number of the stiffness matrix
associated with the bilinear form~\eqref{eq:weak-cutfem-formulation}
can be bounded by $O(h^{-2})$ independently of the position of the
surface $\Gamma$ relative to the background mesh~$\mcT_h$.

Let $\{\phi_i\}_{i=1}^N$ be the standard piecewise linear basis
functions associated with $\mcT_h$ and thus
$v_h = \sum_{i=1}^N V_i \phi_i$ for $v_h \in V_h$ and expansion
coefficients $V = \{V_i\}_{i=1}^N \in \RR^N$.
The stiffness matrix $\mcA$ is given by the relation
\begin{align}
  ( \mcA V, W )_{\RR^N}  = A_h(v_h, w_h) \quad \foralls v_h,w_h \in
  V_h
  \label{eq:stiffness-matrix}
\end{align}
Recalling the definition of $V_h$ and
proposition~\ref{prop:stability-estimates-Ah}, 
the stiffness matrix $\mcA$ 
clearly is a bijective linear mapping 
$\mcA:\widehat{\RR}^N \to \ker(\mcA)^{\perp}$
where we set $\widehat{\RR}^N = \RR^N /\ker(\mcA)$
to factor out the one-dimensional kernel given by 
$\ker{\mcA} = \spann\{(1,\ldots,1)^{\top}\}$.
The operator norm and condition number of the matrix $\mcA$ are then defined by
\begin{align}
  \| \mcA \|_{\RR^N}
  = \sup_{V \in \widehat{\RR}^N\setminus\bfzero}
  \dfrac{\| \mcA V \|_{\RR^N}}{\|V\|_N}
\quad \text{and}
\quad
  \kappa(\mcA) = \| \mcA \|_{\RR^N} \| \mcA^{-1} \|_{\RR^N}
  \label{eq:operator-norm-and-condition-number-def}
\end{align}
respectively.
Following the approach in~\citet{ErnGuermond2006}, 
a bound for the condition number can be derived 
by combining the well-known estimate
\begin{align}
  h^{d/2} \| V \|_{\RR^N}
  \lesssim \| v_h \|_{L^2(N_h)}
  \lesssim
  h^{d/2} \| V \|_{\RR^N}
  \label{eq:mass-matrix-scaling}
\end{align}
which holds for any quasi-uniform mesh $\mcT_h$,
with the Poincar\'e-type estimate~\eqref{eq:discrete-poincare-Nh} and
the following inverse estimate:
\begin{lemma} Let $v \in V_{h,0}$ then the following inverse estimate holds
  \label{lem:inverse-estimate-Ah}
  \begin{align}
    \tn v \tn_h \lesssim h^{-3/2} \| v \|_{N_h}
    \label{eq:inverse-estimate-Ah}
  \end{align}
  \begin{proof}
    First, we note that employing the standard inverse
    estimates~\eqref{eq:inverse-estimates-boundary}, we obtain
    \begin{align}
      \tn v \tn_{\mcF_h}^2 \lesssim h^{-3} \| v \|_{N_h}^2
    \end{align}
    Now, since $\nabla v|_T$ is constant, an application of the inverse
    estimates~\eqref{eq:inverse-estimate-cut-v-on-K} and
    \eqref{eq:inverse-estimate-grad} gives
    \begin{align}
      \| \Psh \nabla v \|_{\mcK_h}^2
      \lesssim 
      \| \nabla v \|_{\mcK_h}^2
      \lesssim
      h^{-1}\| \nabla v \|_{N_h}^2
      \lesssim
      h^{-3}\| v \|_{N_h}^2
    \end{align}
    Similarly, recalling~\eqref{eq:inverse-estimate-cut-on-E}, the co-normal flux  
    and edge-related jump terms can be bounded via 
    \begin{gather}
      \| n_E \cdot \nablash v \|_{\mcE_h}^2
      \lesssim
      h^{-1}\| \nabla v \|_{\mcF_h}^2
      \lesssim 
      h^{-2}\| \nabla v \|_{N_h}^2
      \lesssim
      h^{-3} \| v \|_{N_h}^2
      \\
      h^{-1} \| \jump{v}\|_{\mcE_h}^2
      \lesssim
      h^{-2}\| v\|_{\mcF_h}^2
      \lesssim
      h^{-3} \| v \|_{N_h}^2
    \end{gather}
  which concludes the proof.
  \end{proof}
\end{lemma}
We are now in the position to prove the main result of this section:
\begin{theorem} 
  \label{thm:condition-number-estimate}
  The condition number of the stiffness matrix satisfies
  the estimate
\begin{equation}
\kappa( \mcA )\lesssim h^{-2}
\end{equation}
where the hidden constant depends only on the quasi-uniformness
parameters.
\end{theorem}
\begin{proof} We need to bound $\| \mcA \|_{\RR^N}$ and $\| \mcA^{-1} \|_{\RR^N}$. 
  To derive a bound for $\| \mcA \|_{\RR^N}$, we first observe that for $w \in V_h$,
\begin{equation}
\tn w \tn_h 
\lesssim h^{-3/2} \| w \|_{\Omega_h}
\lesssim h^{(d-3)/2}\|W\|_{\RR^N}
\end{equation}
where we successively used the inverse estimate~\eqref{eq:inverse-estimate-Ah}
and equivalence~\eqref{eq:mass-matrix-scaling}.
Consequently,
\begin{align}
  \| \mcA V\|_{\RR^N} &= \sup_{W \in \RR^N } 
  \frac{( \mcA V, W)_{\RR^N}}{\| W \|_{\RR^d}}
= \sup_{w \in V_h }  \frac{A_h(v,w)}{\tn w \tn_h} \frac{\tn w \tn_h}{
\| W \|_{\RR^N}}
\lesssim h^{(d-3)/2} \tn v \tn_h 
\lesssim h^{d-3}| V|_N
\end{align}
and thus by the definition of the operator norm, we have
$
\| \mcA \|_{\RR^N} \lesssim h^{d-3}
$.
Next we turn to the estimate of $\| \mcA^{-1}\|_{\RR^N}$.
Starting from \eqref{eq:mass-matrix-scaling} and combining the Poincar\'e
inequality~\eqref{eq:discrete-poincare-Nh} with the stability
estimates~\eqref{eq:coercivity}--\eqref{eq:continuity} and
a Cauchy Schwarz inequality, we arrive at the following chain of
estimates:
\begin{align}
  \| V \|^2_{\RR^N} 
  \lesssim h^{-d} \| v \|^2_{\Omega_h} 
  \lesssim h^{1-d} \tn v \tn_h^2
  \lesssim h^{1-d} A_h(v,v) 
  = h^{1-d} (V, \mcA V)_{\RR^N}
  \lesssim h^{1-d} \| V \|_{\RR^N} \| \mcA V \|_{\RR^N}
\end{align}
and hence $\| V \|_{\RR^N} \lesssim h^{1-d}\| \mcA V\|_{\RR^N}$. 
Now setting $ V = \mcA^{-1} W$ we conclude that
 $
 \| \mcA^{-1}\|_{\RR^N} \lesssim h^{1-d}
 $
and combining estimates for $\| \mcA\|_{\RR^N}$ and $\| \mcA^{-1}\|_{\RR^N}$ the theorem follows.
\end{proof}

\section{Numerical Results}
\label{sec:numerical-results}
We conclude this paper with two numerical studies.
First, we corroborate the theoretical a priori estimates presented in
Section~\ref{ssec:apriori-est} with two convergence experiments. 
The second study serves to illustrate the effect of the different
stabilization terms in~\eqref{eq:jh-def} on the sensitivity of the
condition number with respect to the surface positioning in the
background mesh.

\subsection{Convergence Rate}
To examine the theoretically expected convergence rates,
we consider two numerical examples for the Laplace-Beltrami-type
problem
\begin{align}
  -\Delta_{\Gamma} u + u = f \quad \text{on } \Gamma
  \label{eq:laplace-beltrami-type-problem}
\end{align}
with given analytical reference solution $u$
and surface $\Gamma = \{ x \in \RR^3 : \phi(x) = 0 \}$
defined by 
a known smooth scalar function $\phi$
with $\nabla \phi(x) \neq 0 \,\forall x \in \Gamma$.
The corresponding right-hand side $f$
can be computed using the following representation of the
Laplace-Beltrami operator
\begin{align}
    \Delta_{\Gamma} u = \Delta u - n_{\Gamma} \cdot \nabla \otimes\nabla u \,
    n_{\Gamma} - \textrm{tr}(\nabla n_{\Gamma}) \nabla u \cdot
    n_{\Gamma}
\end{align}
For the first test example we chose
\begin{equation}
  \left\{
    \begin{aligned}
      u_1 &= 
      \sin\left(\dfrac{\pi x }{2}\right)
      \sin\left(\dfrac{\pi y }{2}\right)
      \sin\left(\dfrac{\pi z }{2}\right)
      \\
      \phi_1 &=  x^2 + y^2 + z^2 - 1
    \end{aligned}
  \right.
\end{equation}
while in the second example, we consider the problem defined by
\begin{equation}
  \left\{
    \begin{aligned}
      u_2 &=  xy - 5y + z + xz
      \\
      \phi_2 &=  
      (x^2 - 1)^2
      +
      (y^2 - 1)^2
      +
      (z^2 - 1)^2
      +
      (x^2 + y^2 - 4)^2
      +
      (x^2 + z^2 - 4)^2
      \\
      &\quad +
      (y^2 + z^2 - 4)^2
      - 16
    \end{aligned}
  \right.
\end{equation}
Starting from a structured mesh $\widetilde{\mcT}_0$ for $\Omega =
[-a,a]^3$ with $a$ large enough such $ \Gamma \subseteq \Omega$,
a sequence of meshes
$\{\mcT_k\}_{k=0}^5$ is generated 
for each test case by successively refining
$\widetilde{\mcT}_0$ and ex\-tracting the corresponding active
background mesh as defined
by~\eqref{eq:narrow-band-mesh}. 
Based on the manufactured exact solutions, the experimental order of
convergence (EOC) is then calculated by
\begin{align*}
    \text{EOC}(k) = \dfrac{\log(E_{k-1}/E_{k})}{\log(2)}
\end{align*}
where $E_k$ denotes the error of the numerical
solution $u_k$ measured in a specified norm and computed at refinement level $k$.
In our convergence studies, both $\| \cdot \|_{H^1(\Gamma_h)}$
and $\|\cdot\|_{L^2(\Gamma_h)}$ as well as $\| \cdot
\|_{L^{\infty}(N_h)}$ are used to compute $E_k$.
For the two test cases, the resulting errors for the sequence of
refined meshes are summarized in
Table~\ref{tab:convergence-rates-example-1} and 
Table~\ref{tab:convergence-rates-example-2}, respectively.
The observed EOC confirms the first-order and second-order
convergences rates as predicted by Theorem~\ref{thm:aprioriest}.
Additionally, we observe an optimal, second-order convergence in
the $L^{\infty}(N_h)$-norm.
Repeating the convergence study
with $\beta_E = 0$, $\beta_F = 500$ and
$\gamma$ as before yields almost identical error reduction rates
for the simplified version of~\eqref{eq:ah-def}
described in Remark~\ref{rem:simplified-ah} (see Table~\ref{tab:convergence-rates-example-2}).
The computed solutions are visualized in
Figure~\ref{fig:solution-example-1}.

\begin{figure}[htb]
  \begin{center}
    \includegraphics[width=0.49\textwidth]{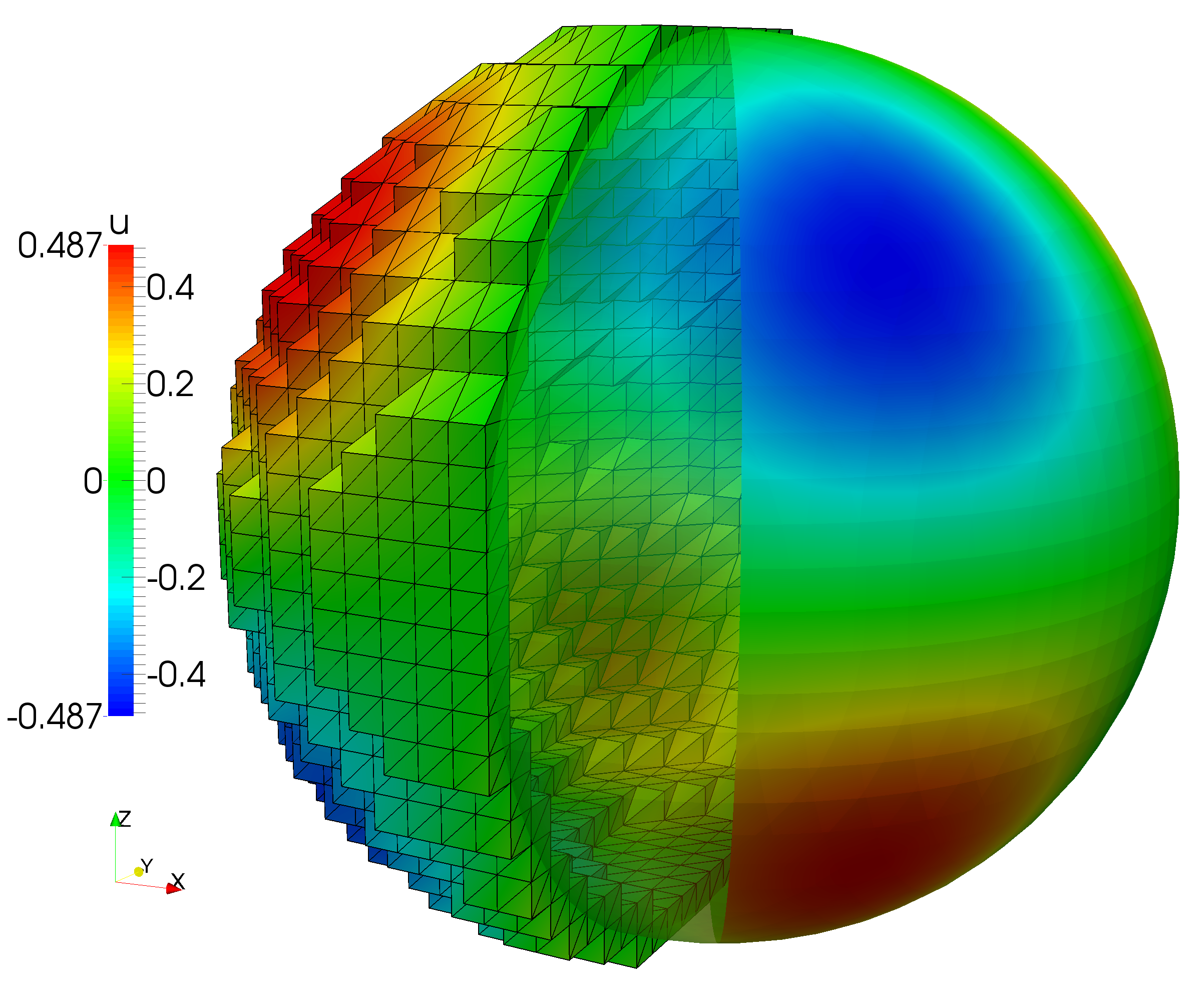}
    \hspace{0.5ex}
    \includegraphics[width=0.49\textwidth]{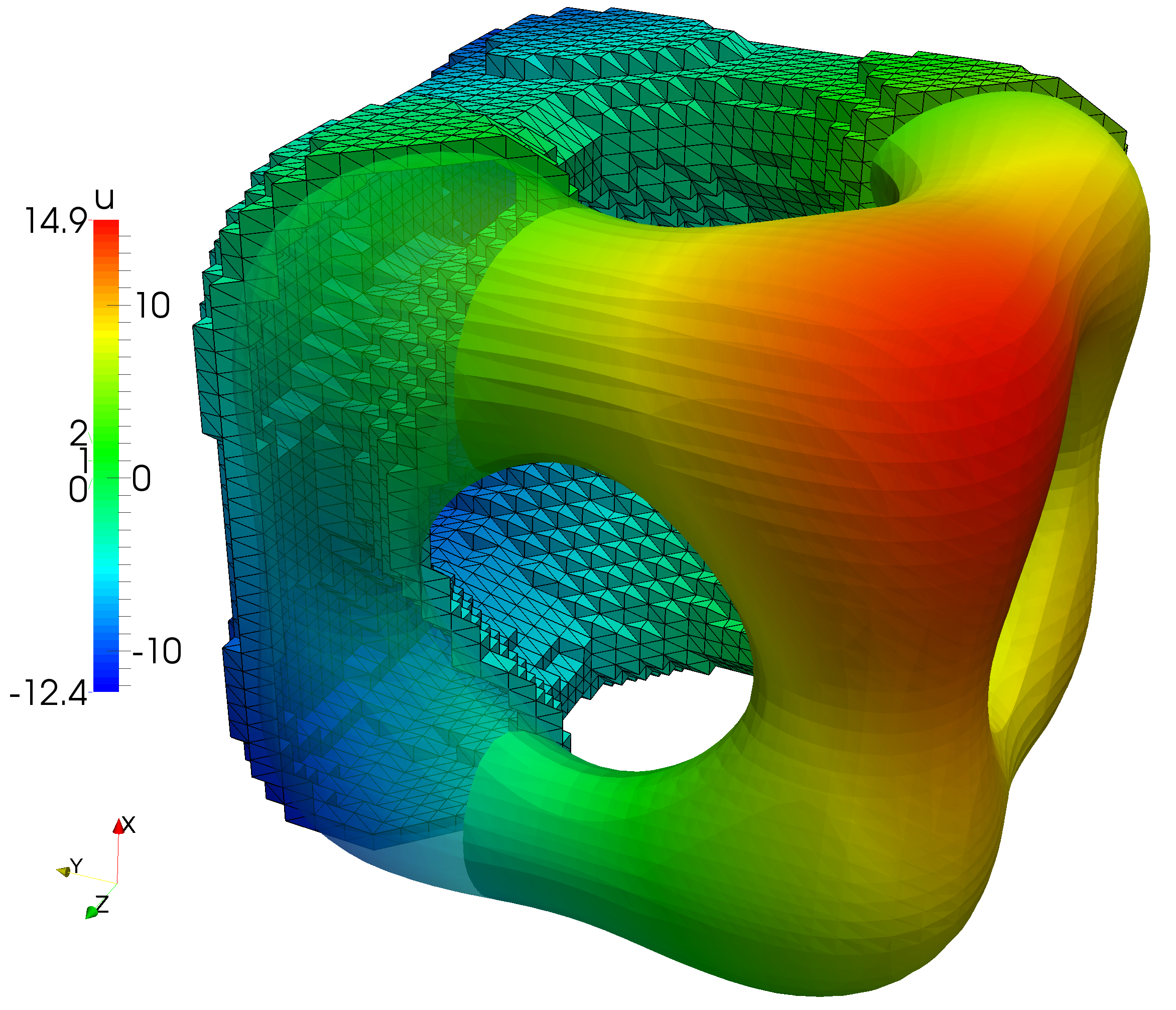}
    \caption{Computed solutions for example 1 (left) and example 2 (right). 
      The left part of each solution plot shows the
      approximation $u_h$ as computed on the active mesh $\mcT_h$ while the right part 
      displays $u_h$ restricted to the surface mesh $\mcK_h$.}
    \label{fig:solution-example-1}
  \end{center}
\end{figure}

\begin{table}[htb]
  \footnotesize
  \centering
  \begin{center}
    \begin {tabular}{cr<{\pgfplotstableresetcolortbloverhangright }@{}l<{\pgfplotstableresetcolortbloverhangleft }cr<{\pgfplotstableresetcolortbloverhangright }@{}l<{\pgfplotstableresetcolortbloverhangleft }cr<{\pgfplotstableresetcolortbloverhangright }@{}l<{\pgfplotstableresetcolortbloverhangleft }c}%
\toprule Level $k$&\multicolumn {2}{c}{$\|u_k - u \|_{H^1(\Gamma _h)}$}&EOC&\multicolumn {2}{c}{$\|u_k - u \|_{L^2(\Gamma _h)}$}&EOC&\multicolumn {2}{c}{$\|u_k - u \|_{L^\infty (N_h)}$}&EOC\\\midrule %
\pgfutilensuremath {0}&$1.27$&$\cdot 10^{0}$&--&$2.49$&$\cdot 10^{-1}$&--&$1.44$&$\cdot 10^{-1}$&--\\%
\pgfutilensuremath {1}&$6.87$&$\cdot 10^{-1}$&\pgfutilensuremath {0.89}&$7.28$&$\cdot 10^{-2}$&\pgfutilensuremath {1.78}&$5.10$&$\cdot 10^{-2}$&\pgfutilensuremath {1.50}\\%
\pgfutilensuremath {2}&$3.73$&$\cdot 10^{-1}$&\pgfutilensuremath {0.88}&$1.93$&$\cdot 10^{-2}$&\pgfutilensuremath {1.91}&$1.35$&$\cdot 10^{-2}$&\pgfutilensuremath {1.92}\\%
\pgfutilensuremath {3}&$1.74$&$\cdot 10^{-1}$&\pgfutilensuremath {1.10}&$4.81$&$\cdot 10^{-3}$&\pgfutilensuremath {2.01}&$3.70$&$\cdot 10^{-3}$&\pgfutilensuremath {1.86}\\%
\pgfutilensuremath {4}&$8.65$&$\cdot 10^{-2}$&\pgfutilensuremath {1.01}&$1.20$&$\cdot 10^{-3}$&\pgfutilensuremath {2.01}&$9.42$&$\cdot 10^{-4}$&\pgfutilensuremath {1.97}\\%
\pgfutilensuremath {5}&$4.34$&$\cdot 10^{-2}$&\pgfutilensuremath {0.99}&$3.01$&$\cdot 10^{-4}$&\pgfutilensuremath {1.99}&$2.39$&$\cdot 10^{-4}$&\pgfutilensuremath {1.98}\\\bottomrule %
\end {tabular}%

  \end{center}
  \vspace{1ex}
  \caption{Convergence rates for example 1 with $\beta_E = \beta_F =
    50$  and $\gamma = 0.01$.}
  \label{tab:convergence-rates-example-1}
\end{table}

\begin{table}[htb]
  \footnotesize
  \centering
  \begin{center}
    \begin {tabular}{cr<{\pgfplotstableresetcolortbloverhangright }@{}l<{\pgfplotstableresetcolortbloverhangleft }cr<{\pgfplotstableresetcolortbloverhangright }@{}l<{\pgfplotstableresetcolortbloverhangleft }cr<{\pgfplotstableresetcolortbloverhangright }@{}l<{\pgfplotstableresetcolortbloverhangleft }c}%
\toprule Level $k$&\multicolumn {2}{c}{$\|u_k - u \|_{H^1(\Gamma _h)}$}&EOC&\multicolumn {2}{c}{$\|u_k - u \|_{L^2(\Gamma _h)}$}&EOC&\multicolumn {2}{c}{$\|u_k - u \|_{L^\infty (N_h)}$}&EOC\\\midrule %
\pgfutilensuremath {0}&$2.21$&$\cdot 10^{1}$&--&$1.48$&$\cdot 10^{1}$&--&$4.56$&$\cdot 10^{0}$&--\\%
\pgfutilensuremath {1}&$7.33$&$\cdot 10^{0}$&\pgfutilensuremath {1.59}&$1.84$&$\cdot 10^{0}$&\pgfutilensuremath {3.01}&$8.54$&$\cdot 10^{-1}$&\pgfutilensuremath {2.42}\\%
\pgfutilensuremath {2}&$3.12$&$\cdot 10^{0}$&\pgfutilensuremath {1.23}&$5.41$&$\cdot 10^{-1}$&\pgfutilensuremath {1.77}&$2.26$&$\cdot 10^{-1}$&\pgfutilensuremath {1.92}\\%
\pgfutilensuremath {3}&$1.61$&$\cdot 10^{0}$&\pgfutilensuremath {0.95}&$1.16$&$\cdot 10^{-1}$&\pgfutilensuremath {2.22}&$5.38$&$\cdot 10^{-2}$&\pgfutilensuremath {2.07}\\%
\pgfutilensuremath {4}&$7.65$&$\cdot 10^{-1}$&\pgfutilensuremath {1.07}&$2.86$&$\cdot 10^{-2}$&\pgfutilensuremath {2.02}&$1.28$&$\cdot 10^{-2}$&\pgfutilensuremath {2.07}\\%
\pgfutilensuremath {5}&$3.78$&$\cdot 10^{-1}$&\pgfutilensuremath {1.02}&$7.18$&$\cdot 10^{-3}$&\pgfutilensuremath {2.00}&$3.32$&$\cdot 10^{-3}$&\pgfutilensuremath {1.95}\\\bottomrule %
\end {tabular}%

  \end{center}
  \begin{center}
    \begin {tabular}{cr<{\pgfplotstableresetcolortbloverhangright }@{}l<{\pgfplotstableresetcolortbloverhangleft }cr<{\pgfplotstableresetcolortbloverhangright }@{}l<{\pgfplotstableresetcolortbloverhangleft }cr<{\pgfplotstableresetcolortbloverhangright }@{}l<{\pgfplotstableresetcolortbloverhangleft }c}%
\toprule Level $k$&\multicolumn {2}{c}{$\|u_k - u \|_{H^1(\Gamma _h)}$}&EOC&\multicolumn {2}{c}{$\|u_k - u \|_{L^2(\Gamma _h)}$}&EOC&\multicolumn {2}{c}{$\|u_k - u \|_{L^\infty (N_h)}$}&EOC\\\midrule %
\pgfutilensuremath {0}&$2.20$&$\cdot 10^{1}$&--&$1.49$&$\cdot 10^{1}$&--&$4.59$&$\cdot 10^{0}$&--\\%
\pgfutilensuremath {1}&$7.22$&$\cdot 10^{0}$&\pgfutilensuremath {1.61}&$1.80$&$\cdot 10^{0}$&\pgfutilensuremath {3.05}&$8.20$&$\cdot 10^{-1}$&\pgfutilensuremath {2.48}\\%
\pgfutilensuremath {2}&$3.09$&$\cdot 10^{0}$&\pgfutilensuremath {1.22}&$5.38$&$\cdot 10^{-1}$&\pgfutilensuremath {1.74}&$2.23$&$\cdot 10^{-1}$&\pgfutilensuremath {1.88}\\%
\pgfutilensuremath {3}&$1.62$&$\cdot 10^{0}$&\pgfutilensuremath {0.93}&$1.16$&$\cdot 10^{-1}$&\pgfutilensuremath {2.22}&$5.27$&$\cdot 10^{-2}$&\pgfutilensuremath {2.08}\\%
\pgfutilensuremath {4}&$7.64$&$\cdot 10^{-1}$&\pgfutilensuremath {1.09}&$2.85$&$\cdot 10^{-2}$&\pgfutilensuremath {2.02}&$1.25$&$\cdot 10^{-2}$&\pgfutilensuremath {2.07}\\%
\pgfutilensuremath {5}&$3.78$&$\cdot 10^{-1}$&\pgfutilensuremath {1.02}&$7.14$&$\cdot 10^{-3}$&\pgfutilensuremath {2.00}&$3.29$&$\cdot 10^{-3}$&\pgfutilensuremath {1.93}\\\bottomrule %
\end {tabular}%

  \end{center}
  \vspace{1ex}
  \caption{Convergence rates for example 2 with $\beta_E = \beta_F =
    50$ (top) and $\beta_E = 0$, $\beta_F = 500$ (bottom) and
  $\gamma = 0.01$ in both cases.}
  \label{tab:convergence-rates-example-2}
\end{table}

\subsection{Condition Number Tests}
Finally, we numerically examine the mesh-size dependency of the condition number
of our proposed method and study how the positioning of the  surface 
in the background mesh affects the condition
number.

Let $\{\mcT_k\}_{k=0}^3$ be a sequence of successively refined
tessellations of $\Omega = [-1.6, 1.6]^3$ with mesh size $h = 3.2/5 \cdot 2^{-k}$.
On each refinement level $k$, we generate a family of surfaces
$\{\Gamma_{\delta}\}_{0\leqslant\delta\leqslant 1}$  
by translating the unit-sphere $S^2 = \{ x \in \RR^3 : \| x \| = 1 \}$
along the diagonal $(h,h,h)$; that is,
$\Gamma_{\delta} = S^2 + \delta_0 (h,h,h)$ with $\delta \in [0,1]$.
For $\delta = l/500$, $l=0,\ldots,500$,
we compute the condition number
$\kappa_{\delta}(\mcA)$ as the ratio of the absolute value of the largest (in modulus) and smallest (in modulus), non-zero eigenvalue.
For each refinement level $k$, the resulting condition numbers are plotted 
in Figure~\ref{fig:condition_number} as a function of $\delta$.
We observe that the position of $\Gamma$ relative to the background mesh
$\mcT_k$ has very little effect on the condition number
when the stabilization parameters $\beta_F$ and $\gamma$
are chosen sufficiently large.
In contrast, the condition number is highly sensitive and clearly
unbounded as a function of $\delta$ if we set either penalty
parameter in~\eqref{eq:jh-def} to $0$ as the corresponding plots in
Figure~\ref{fig:condition_number} reveal.
Finally, scaling the largest computed condition number on each refinement
level $k$ with $k^{-2}$ confirms the theoretically proven $O(h^{-2})$ bound, see Table~\ref{tab:scaled-condition-number}.

To further elucidate the importance of the various stabilization terms in~\eqref{eq:jh-def},
we conduct a second numerical experiment inspired by the results
presented in~\citet{OlshanskiiReusken2010}.
\citet{OlshanskiiReusken2010} show that diagonal preconditioning
yields robust condition number bounds and cures the discrete system
from being severely ill-conditioned
when a \emph{continuous} piecewise linear ansatz space is used,
see also Figure~\ref{fig:condition_number}.
Thus, no stabilization is needed and $\gamma$, as the only relevant
penalty parameter in the continuous case, can be set to $0$.
Repeating the same experiment for our proposed cut discontinuous
Galerkin method with either penalty parameter $\beta_F$ or $\gamma$
deactivated shows that the same conclusion is not true if
\emph{discontinuous} finite element functions are employed.
The reason lies in the construction of the approximation space by
restricting the finite element space defined on the background mesh to the
surface. 
For each element $T \in \mcT_h$, restricting the shape
functions defined on $T$ to the surface part $K = \Gamma_h \cap T$
clearly yields a set of locally linearly dependent functions on $K$.
With the basis functions stretching over several elements, this effect
is usually counteracted in the case of continuous finite element
functions, since curvature effects lead to different linear
dependencies on each element. 
Thus the resulting set of finite element functions on $\mcK_h$ is usually only
``nearly'' linearly dependent and therefore a well-conditioned system
can be obtained by means of proper preconditioning.
However, in the discontinuous case, restricting the finite element basis on $\mcT_h$
to the surface produces a globally linearly dependent set of discrete functions,
even in the presence of highly curved surfaces,
as interelement continuity is imposed only weakly via~\eqref{eq:jh-def}.
We also observe that if we enforce nearly inter-element continuity of the discrete functions 
by choosing a very large penalty parameter, e.g., $\beta_F = 5e6$,
diagonal preconditioning yields robust, albeit large, bounds for the condition number,
see Figure~\ref{fig:condition_number}. 

\begin{figure}[htb]
  \begin{center}
    \includegraphics{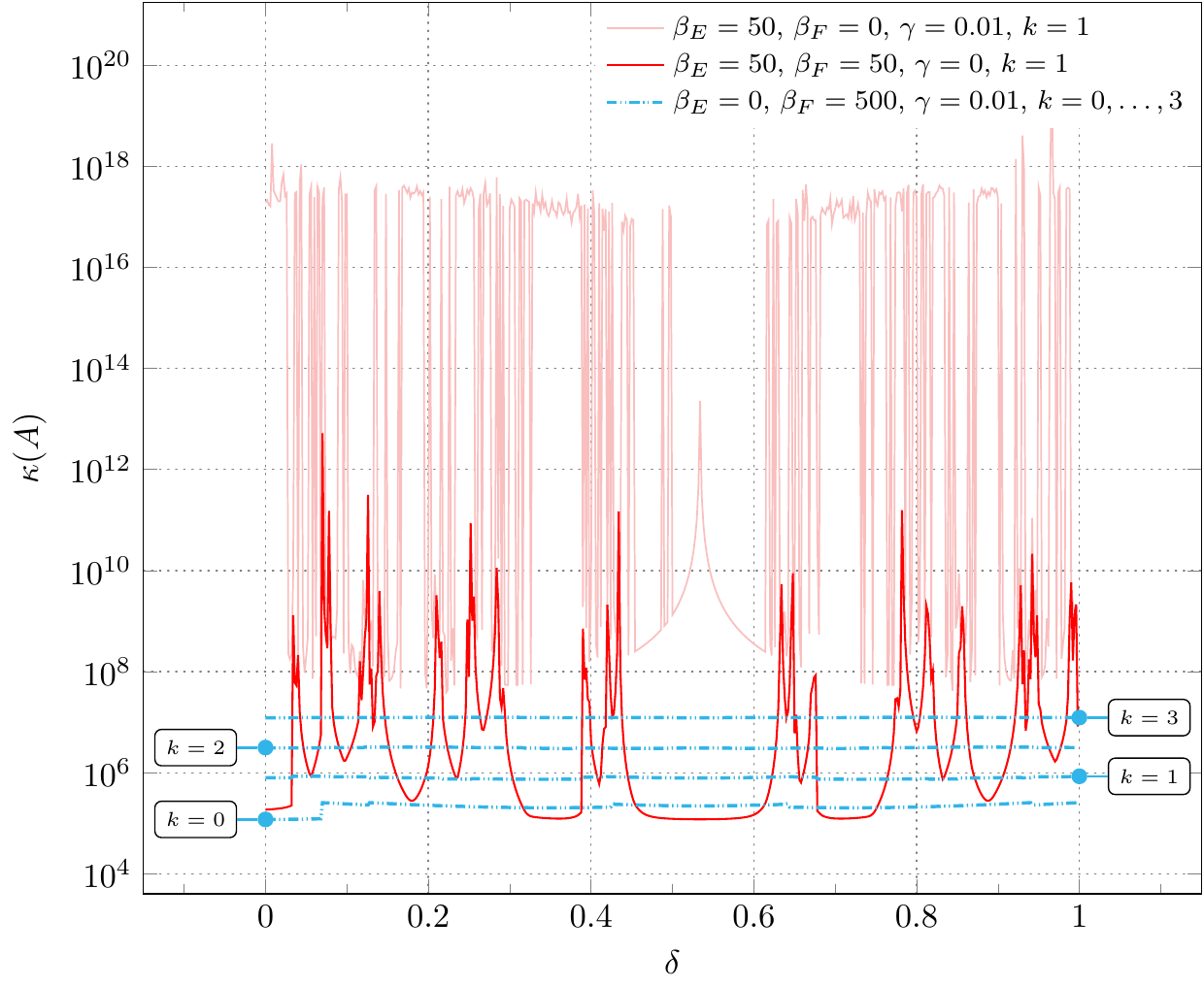}
    \includegraphics{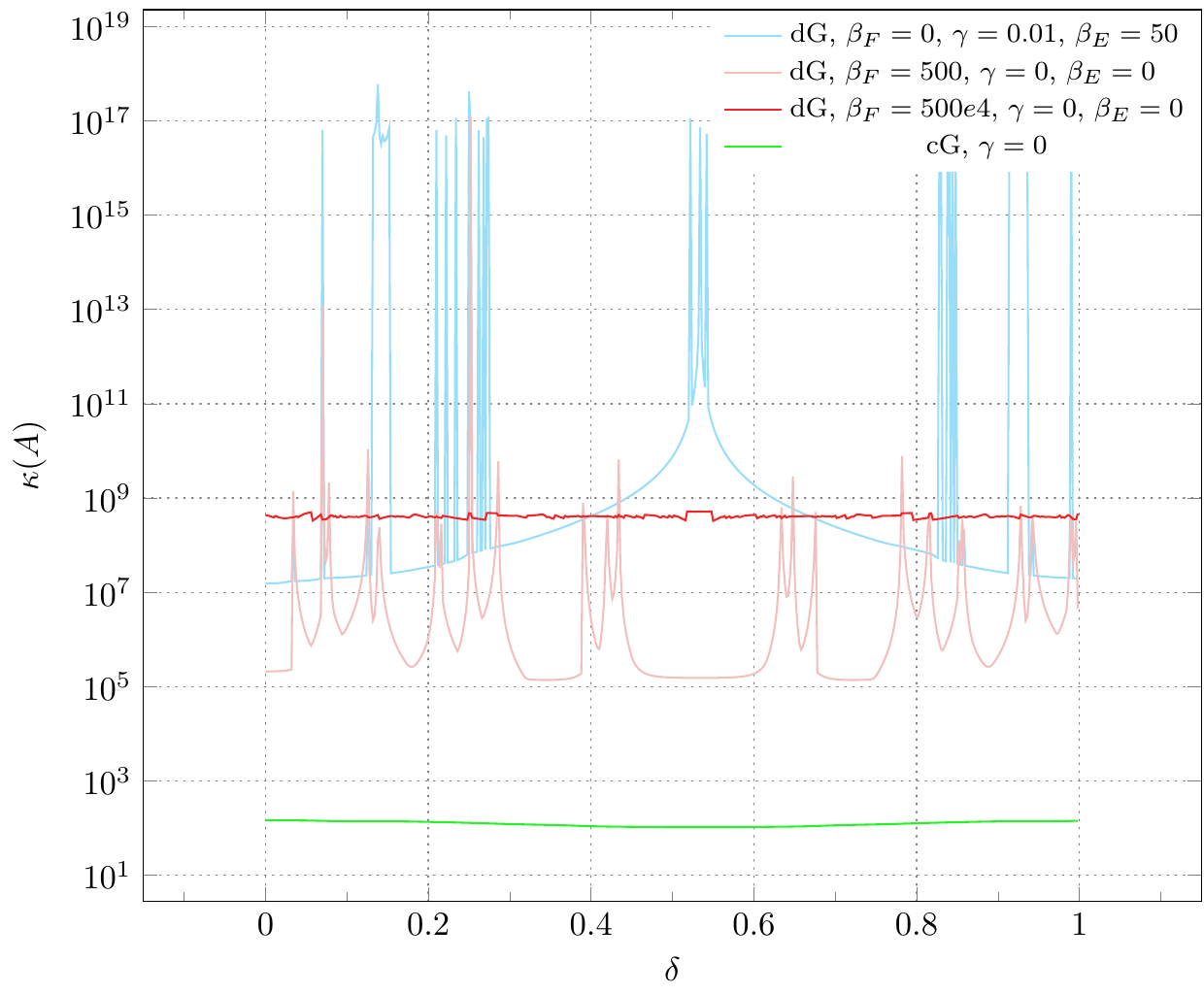}
    \caption{Condition numbers plotted as a function of the position parameter $\delta$.
      Top: With either $\beta_F=0$ or $\gamma = 0$, the condition number
      is highly sensitive to the surface positioning in the background mesh.
      With both penalties activated, the condition number is robust and
      the edge-based penalty term $\beta_E$ can be omitted.
      Bottom: Diagonally preconditioned surface CutFEMs. For an unstabilized continuous Galerkin (cG) method,
      preconditioning gives robust condition numbers while for our discontinuous CutFEM (dG),
      the preconditioned system is still highly dependent on the surface position
      if either $\beta_F = 0$ and $\gamma = 0$.
  }
    \label{fig:condition_number}
  \end{center}
\end{figure}

\begin{table}[htb]
  \footnotesize
\begin{center}
  \begin{tabular}[c]{l r r r r}
    \toprule
    $k$ & 1 & 2 & 3 & 4
    \\
    \midrule
    $k^{-2} \max_{\delta}\{\kappa_{\delta}(\mcA)\}$ 
    &
    7370.36& 7167.30& 7205.85& 6271.44
    \\
    \bottomrule
  \end{tabular}
  \vspace{2ex}
  \caption{Maximum of scaled condition numbers for each refinement
  level $k$.}
  \label{tab:scaled-condition-number}
\end{center}
\end{table}

\clearpage

\section*{Acknowledgements}
This research was supported in part by EPSRC, UK, Grant
No. EP/J002313/1, the Swedish Foundation for Strategic Research Grant
No.\ AM13-0029, the Swedish Research Council Grants Nos.\ 2011-4992,
2013-4708, and Swedish strategic research programme eSSENCE.

\bibliographystyle{plainnat}
\bibliography{bibliography}

\end{document}